\documentclass{amsart}
\usepackage{chngcntr}
\usepackage{apptools}
\usepackage{color}
\usepackage{amsthm,amssymb,verbatim}
\usepackage{graphicx}
\usepackage{enumerate}
\usepackage[normalem]{ulem}
\usepackage{calrsfs}
\DeclareMathAlphabet{\pazocal}{OMS}{zplm}{m}{n}

\DeclareMathSymbol{\mh}{\mathord}{operators}{`\-}

\newcommand{\bone}{_\textnormal{big, 1}}
\newcommand{\btwo}{_\textnormal{big, 2}}
\newcommand{\sone}{_\textnormal{small, 1}}
\newcommand{\stwo}{_\textnormal{small, 2}}

\definecolor{pp}{rgb}{.5,0,.8}
\definecolor{rr}{rgb}{.8,0,.3}

\newcommand{\dbone}{\textnormal{d}\bone}
\newcommand{\dbtwo}{\textnormal{d}\btwo}
\newcommand{\dsone}{\textnormal{d}\sone}
\newcommand{\dstwo}{\textnormal{d}\stwo}

\newcommand{\Ff}{\mathcal F}

\newcommand{\Rr}{\mathcal R}

\newcommand{\Sc}{\pazocal{S}}   
 
 \newcommand{\Cc}{\mathcal C}
 \newcommand{\Dd}{\mathcal{D}}
 
 \newcommand{\Ee}{\mathcal{E}}

 \newcommand{\genus}{\operatorname{genus}}

 \newcommand{\Ll}{\mathcal{L}}

 \newcommand{\Mm}{\mathcal{M}}
  \newcommand{\MM}{\mathcal{M}}
  \newcommand{\nn}{\mathbf{n}}
\newcommand{\kk}{\mathbf{k}}

  \newcommand{\Ss}{\mathbf{S}}
 \newcommand{\Bb}{\mathcal B}
 \newcommand{\RR}{\mathbf{R}}  

 \newcommand{\wtN}{\widetilde N}
\newcommand{\wtM}{\widetilde M}
 \newcommand{\ZZ}{\mathbf{Z}}  
    
  \newcommand{\Div}{\operatorname{Div}}
  
    \newcommand{\dist}{\operatorname{dist}}

 \newcommand{\area}{\operatorname{area}}
 \newcommand{\eps}{\varepsilon}
 \newcommand{\Tan}{\operatorname{Tan}}

\newcommand{\ee}{\mathbf e}

\newcommand{\Index}{\operatorname{index}}
\newcommand{\Hh}{\mathcal{H}}
\usepackage{amsthm}
\usepackage[normalem]{ulem}

\input epsf
\def\begfig {
\begin{figure}
\small }
\def\endfig {
\normalsize
\end{figure}
}

    \newtheorem{theorem}    {Theorem}      [section]
    \newtheorem{lemma}      [theorem]       {Lemma}
    \newtheorem{corollary}  [theorem]     {Corollary}
    \newtheorem{proposition}       [theorem]       {Proposition}
    \newtheorem{claim}{Claim}
    \newtheorem*{theorem*}{Theorem}
    \newtheorem*{definition*}{Definition}
     \newtheorem*{remark*}{Remark}

    \theoremstyle{definition}
    \newtheorem{definition}  [theorem] {Definition}
     
    \theoremstyle{definition}
    \newtheorem{remark}   [theorem]       {Remark}
    
\usepackage{enumerate}
\usepackage{hyperref}
\usepackage[alphabetic, msc-links, backrefs]{amsrefs}
\numberwithin{equation}{section}

\title[Counterexamples to Ilmanen's Genus-Reduction Conjecture]{Expanders 
for Mean Curvature Flow and Counterexamples to Ilmanen's Genus-Reduction Conjecture}
\author[D. Hoffman]{\textsc{D. Hoffman}}

\address{David Hoffman\newline
 Department of Mathematics\newline
 Stanford University \newline
   Stanford, CA 94305, USA\newline
{\sl E-mail address:} {\bf dhoffman@stanford.edu}}

\author[F. Mart\'in]{\textsc{F. Mart\'in}}

\address{Francisco Mart\'in\newline
Departamento de Geometr\'ia y Topolog\'ia  \newline
Instituto de Matem\'aticas de Granada (IMAG) \newline
Universidad de Granada\newline
18071 Granada, Spain\newline
{\sl E-mail address:} {\bf fmartin@ugr.es}
}
\author[B. White]{\textsc{B. White}}

\address{Brian White\newline
Department of Mathematics \newline
 Stanford University \newline 
  Stanford, CA 94305, USA\newline
{\sl E-mail address:} {\bf bcwhite@stanford.edu}
}
\thanks{F. Mart\'in was partially supported by the grant PID2024-156031NB-I00 and by the  IMAG–Maria de Maeztu grant CEX2020-001105-M, both funded by MICIU/AEI/10.13039/501100011033.}

\date{May 10, 2026}

\begin{document}
\begin{abstract}
We construct new expanders for mean curvature flow that are smoothly asymptotic to cones arising from
certain shrinkers.  For each such cone, we prove the existence of expanders of arbitrarily large genus.
 Thus, for a fixed incoming shrinker, the genus of the outgoing expander can be chosen much larger than the genus before the singularity, contrary to Ilmanen's genus-reduction conjecture.
\end{abstract}

\maketitle

\tableofcontents

\section{Introduction}

In 1995, Tom Ilmanen made the following conjecture:

\begin{quote}{\bf Genus Reduction Conjecture}. \it Let $M(t)$ be any mean curvature flow in
$\RR^3$. Then the genus of $M(t)$ strictly decreases at the moment of a singularity,
unless the singularity is a neckpinch or shrinking sphere.\end{quote}
See~\cite{ilmanen1}*{p. 23}.

It follows from work of
 Chodosh-Choi-Schulze~\cite{ChChS} and 
 Bamler-Kleiner~\cite{bamler-kleiner} that for any compact, smoothly embedded surface $M$ in $\RR^3$, the genus reduction conjecture holds if $M(\cdot)$ is the innermost or the outermost mean curvature flow starting from $M$.  See
 Theorem~\ref{inner-outer-theorem} in the Appendix.

In this paper, we show that the conjecture is false in general by giving examples in  which the genus after the singularity  is much 
bigger than the genus before the singularity. In fact, we show that 
there is an initial surface $M$ for which there is no upper bound on genus after the first singular time $T_{\rm sing}$: for every $g<\infty$, there exists a  mean curvature flow $M(t)$ starting with $M$ such that for all $t>T_{\rm sing}$, 
the surface at time $t$ is smooth and has genus $>g$.

All of our examples are  self-similar.  Recall that a surface $M$ in $\RR^3$ is called a {\bf shrinker}
if 
\begin{equation}
\label{first-flow}
  t\mapsto |t|^{1/2}M
  \quad
  (t<0)
\end{equation}
is a mean curvature flow.
If a shrinker $M$ has no boundary, then  the flow~\eqref{first-flow} extends to a Brakke flow
\begin{equation}
\label{shrinking-flow}
t
\mapsto
\begin{cases}
|t|^{1/2}M  &(t<0), \\
C           &(t=0),
\end{cases}
\end{equation}
where $C$ is a cone, i.e., is invariant under positive dilations.
In this paper, we will only consider shrinkers $M$ that are smoothly embedded and cones $C$ that are smooth  surfaces away from the vertex.
For such $M$ and $C$,  \eqref{shrinking-flow} is a mean curvature flow if and only $M$ is smoothly asymptotic to $C$ at infinity.

A surface $E$ in $\RR^3$ is called an {\bf expander} if
\begin{equation}
\label{expanding-flow}
   t\mapsto t^{1/2}E
   \quad
   (t>0)
\end{equation}
is a mean-curvature flow,
or, equivalently, if the mean curvature of $E$ at each point $p$ is 
\begin{equation}
\label{expander-equation}
H(p)=\frac12p^\perp = \frac12(p\cdot\nu)\nu,
\end{equation}
where $\nu$ is a unit normal to $M$ at $p$.

Note that if $M$ and $E$ are smoothly asymptotic to the same cone $C$, then the flows~\eqref{shrinking-flow} and~\eqref{expanding-flow} fit together to form an  eternal mean curvature flow
\begin{equation}
\label{combined-flow}
t
\mapsto
\begin{cases}
|t|^{1/2}M  &(t<0), \\
C  &(t=0), \\
t^{1/2}E,   &(t>0)
\end{cases}
\end{equation}
that is fully self-similar
(i.e., self-similar in both time directions).
  The flow
is smooth except at the spacetime origin.

There are many known examples of shrinkers. 
Thus, to understand self-similar mean curvature flows, it is natural to ask: given a cone $C$, what are the expanders that are  smoothly asymptotic to it?

 In particular,
 in~\cite{hmw-shrinkers}, we proved

\begin{theorem}
\label{shrinker-theorem}
For all sufficiently large $k$, there exists a shrinker $M_k$ of genus $2k-2$ whose associated cone $C_k$ has the following properties:
\begin{enumerate}
\item
\label{property-one}
$C_k$ intersects the plane $\{z=0\}$ in $k$ equally spaced lines through the origin.
\item 
\label{property-two}
$C_k$ 
is invariant under the
  group $G_k$ of isometries of $\RR^3$ defined
  in Section~\ref{preliminaries-section}.
\item
\label{property-three}
$C_k\cap \partial B(0,1)$
consists of three simple closed curves, each of which winds once around $Z$. 
\item
\label{property-four}
$C_k\cap\partial B(0,1)$
lies with distance $\eps_k$
of the plane $\{z=0\}$, where $\eps_k\to 0$ as $k\to\infty$.
\end{enumerate}
\end{theorem}

(See~\cite{sz-shrinkers} for a different existence proof for such shrinkers $S_k$.)

(In fact, each of the three curves in $C_k\cap \{z=0\}$ converges smoothly to the equator $\Ss^2\cap\{z=0\}$ as $k\to\infty$, but our main theorem does not require that.)

Our main theorem is the following:

\begin{theorem} \label{intro-theorem}
There is an $\eps>0$ with the following property.
Suppose $C=C_k$ is a cone with properties~\eqref{property-one}--\eqref{property-four}, where $\eps_k<\eps$.
Suppose that $p$ is an odd prime number and that $n\ge 0$ is an integer.
Then there exist at least $4$ non-congruent expanders $E$ such that 
\begin{enumerate}
\item $E$ is smoothly asymptotic to $C$, 
\item $E$ is $G_k$-invariant.
\item $E$ has genus $p^nk-1$.
\end{enumerate}
\end{theorem}
By choosing $p$ or $n$ (or both) large, we can make the genus of $E$ arbitrarily large.  Thus, if we choose  $M$ to be $M_k$, the genus $2k-2$ shrinker in 
 Theorem~\ref{shrinker-theorem}, then the surfaces $M_k$, $C=C_k$, and $E$
combine to give a mean curvature flow~\eqref{combined-flow} in which the genus after the singularity is vastly larger than the genus before the singularity.

Recently, Shao and Zou \cite{sz-expanders} proved  a special case of
 Theorem~\ref{intro-theorem}: they proved, under somewhat more restrictive conditions on $C_k$, the existence of one expander of genus $k-1$. 
That expander does not provide a counterexample to Ilmanen's conjecture, since the 
corresponding self-similar flow~\eqref{combined-flow}
has genus $2k-2$ for $t<0$ and $k-1$ for $t>0$.

Incidentally, in Theorem~\ref{intro-theorem}, the cone $C$ is allowed to be rather wild.
For example, no upper bound on the lengths of the three components of 
$C\cap\partial B(0,1)$
is assumed. Also, on each component of $C\cap \partial B(0,1)$, the cylindrical coordinate $\theta$ is not required to be monotonic.

Our method also gives  expanders (asymptotic to $C$) with smaller symmetry group, including expanders with smaller genus. 
See Theorem~\ref{j-theorem}.

The proof 
of 
Theorem~\ref{intro-theorem} uses minimal surface theory: a surface $M$ in $\RR^3$ is an expander if and only if it is minimal
with respect to the {\bf expander metric} 
\begin{equation}
\label{expander-metric}
 h:=e^{|p|^2/4}\delta_{ij}. 
\end{equation}
 We prove existence of suitable $h$-minimal surfaces bounded
by $C\cap\partial B(0,R)$
and then let $R\to\infty$ to get complete examples.
Letting $R\to\infty$ is relatively straightforward; see Section~\ref{noescape-section}.  The bulk of the paper is devoted to proving existence for finite $R$.

For finite $R$, we prove existence by proving that a suitable mapping degree is nonzero.  
Unfortunately, except in the case of disks, degrees in minimal surface theory tend to be $0$.
For example, 

\begin{theorem}\cite{hoffman-white-number}*{Theorem~3.2}
\label{old-degree-theorem}
Let $\Sigma$ be a compact,
connected surface with boundary.
Suppose that $N$ is a smooth, compact Riemannian $3$-manifold diffeomorphic to a ball,
that $\partial N$ is strictly mean convex, and that $N$ contains no closed minimal surface.
For a generic smooth embeddding $f$ of $\partial\Sigma$ into $\partial N$, 
\begin{equation}
\label{the-sum}
  \sum_M (-1)^{\Index(M)}
  =
  \begin{cases}
  1 &\text{if $\Sigma$ is a disk}, \\
  0 &\text{if $\Sigma$ is not a disk},
  \end{cases}
\tag{*}
\end{equation}
where the sum is over all smooth embedded minimal surfaces~$M$ diffeomorphic to $\Sigma$ and  having boundary $f(\partial \Sigma)$.
\end{theorem}

We are trying to prove existence of connected surfaces with three boundary components and with genus~$p^nk-1$.
(For this discussion, once should take $p$, $n$, and $k$ as fixed.)
One might think from
 Theorem~\ref{old-degree-theorem} that degree theory would not be useful here. 
To indicate how we get existence in spite of the $0$ 
 in~\eqref{the-sum}, let us consider the analogous
finite dimensional problem.
Suppose one has a smooth, proper map $F:M\to N$, where $M$ and $N$ are oriented manifolds of the same dimension and where $N$ is connected, and suppose one wants to prove (for some $p\in N$) existence
of solutions of 
\begin{equation}
\label{toy-equation}
    F(x)=p.
\end{equation}
If the degree of $F$ is nonzero, then $F$ is surjective, so one gets solutions for every $p$.

What if the degree is $0$?
If one is lucky, $M$ might contain a connected component $M'$ such that $F|M'$ has nonzero mapping degree.
In that case, one still gets solutions of~\eqref{toy-equation}.
Indeed, for each $p\in N$, one would get at least two solutions, since
$\deg(F|(M\setminus M'))= - \deg(F|M')$ would also be nonzero, and thus one would have at least one solution in $M'$ and one solution in $M\setminus M'$.

What if the degree is $0$ and $M$ is connected?
There are two possible ways around the problem:
\begin{enumerate}
\item
\label{way-one}
Suppose there is a connected open subset $U$ of $N$ such that $F^{-1}(U)$ has more than one component.  Then there may be a connected component $W$ of $F^{-1}(U)$ such that 
\[
   F: W\to U
\]
has nonzero mapping degree. In that case, one would get solutions 
of~\eqref{toy-equation} 
 for every $p\in U$.
\item 
\label{way-two}
Suppose a finite group $G$ acts by diffeomorphisms on $M$ and on $N$ in an $F$-equivariant way.
Let $M_G$ and $N_G$ be
the set of $G$-invariant elements of $M$ and of $N$.
The degree of $F:M_G\to N_G$ 
might be nonzero, or, if it is $0$, 
then
$M_G$ might be disconnected and the restriction to one
of the connected components might be nonzero.
\end{enumerate}

More generally, in the discussion above, it is not necessary that $M'$ be a connected component of $M$, or that $W$ be a connected component of $F^{-1}(U)$; it suffices for $M'$ to be a relatively open-and-closed subset of $M$, or for $W$ to be relatively open-and-closed subset of $F^{-1}(U)$.

Now let us return to the infinite-dimensional problem we are interested in.  We wish to solve the equation
\[
  \partial M = \Gamma,
\]
where $\Gamma$ is a disjoint union of three smooth, simple closed curves in  $\partial B(0,R)$, and $M\subset\overline{B(0,R)}$ is a connected, embedded   minimal surface (for the expander
 metric~\eqref{expander-metric} of a specified genus $p^nk-1$. 

We implement~\eqref{way-one} and~\eqref{way-two} by requiring:
\begin{enumerate}[\upshape(i)]
\item  that $\Gamma$ and $M$ be 
 invariant under 
  the group~$G_k$.
\item  that 
$\Gamma$
lies in the region
\begin{equation}
\label{eps-cone}
  |z|\le \eps(x^2+y^2)^{1/2}
\end{equation}
for a suitably small $\eps$, and that each component of $\Gamma$
winds once around $Z$, the $z$-axis.
\end{enumerate}

The group $G_k$ includes reflection in the
 plane~$\{y=0\}$. Thus,  a $G_k$-invariant surface intersects that plane in a system of smooth embedded curves, and we can break our space of surfaces into open and closed subsets based on the behavior of those curves.  In particular, we divide our surfaces into three classes (type 1, type 2, and other), each of which is relatively open and closed.  See Section~\ref{slices-section}.

We also divide our space of surfaces into ``big'' and ``small'' surfaces.  Roughly speaking, a surface is big or small according to how much area it has in the ball $B(0,1/2)$.  In general, any cutoff between ``big'' and ``small'' would be arbitrary, and would not result in open-and-closed subsets of the space of surfaces.
However, if the $\eps$ 
 in~\eqref{eps-cone}
 is sufficiently small,
then the resulting classes of big and small surfaces do turn out to be relatively open and closed.
See Section~\ref{big-small-section}.

Thus we get four classes of surfaces and four mapping degrees: $\dbone$,
$\dbtwo$, $\dsone$, and $\dstwo$.  
We wish to show that each degree is nonzero.  By symmetry,  $\dsone=\dstwo$ and
$\dbone=\dbtwo$. 
Also, $\dbone= - \dsone$
and $\dbtwo= -\dstwo$.
(See Theorem~\ref{big-theorem}).
Thus it suffices to show
 that $\dsone\ne 0$.

One way to prove that a degree
(in minimal surface theory) is nonzero is to deform the boundary to a boundary for which one knows all the minimal surfaces.  Then (if those surfaces have nullity $0$) one can calculate the degree;
it is the number of surfaces of even index minus the number of surfaces of odd index.

Following that strategy, we deform the boundary curve to three horizontal circles in $\partial B(0,R)$ at heights $s$, $0$, and $-s$, where $s$ is small.

We do not have a classification of the small, type 1 minimal surfaces (for the expander metric) with that boundary and with the specified genus $p^nk-1$.  However, we are all able to prove that 
there is a 
 unique maximally symmetric such surface.
Furthermore, assuming a certain ``bumpiness'' condition,
that surface contributes $\pm 1$  to the degree count.
Each other surface (if there are any) occurs in a set of $p^j$ (for some $j\ge 1$) congruent surfaces. Thus those surfaces contribute $\pm p^j$ to the mapping degree.  Consequently, $\dsone$ is congruent to $\pm 1$ mod $p$.
(To make that argument precise, one needs to perturb
the expander metric slightly so that the bumpiness condition does hold.  
 See~\cite{white-degree}.)

\bigskip

Ilmanen made his Genus Reduction Conjecture in 1995, when the subject of mean curvature flow was very much in its infancy.
There was one known result about 
genus reduction:
if $M$ is any compact set in $\RR^3$, if $t\in [0,\infty)\mapsto M(t)$ is the resulting level set flow,
and if $t_1<t_2$ are times
at which $M(t)$ is a smooth embedded surface, then
\begin{equation}
\label{genus-decreases}
   \genus(M(t_1))\ge \genus(M(t_2)).
\end{equation}
See~\cite{white-topology}.

In his article containing the Genus Reduction Conjecture, 
Ilmanen proved that \eqref{genus-decreases} also holds
(assuming $M(t_1)$ and $M(t_2)$ are smooth) for the innermost and outermost
flows starting from a compact,
smoothly embedded  initial 
surface in~$\RR^3$.
See~\cite{ilmanen1}*{Section~E}.
However, until the 2023 work of Bamler and Kleiner~\cite{bamler-kleiner}, it was not known in general if there were any times after the initial singular time at which the surface is smooth.

Note that those positive results assert that the genus cannot increase, whereas the Genus Reduction Conjecture asserts, in certain situations, that the genus strictly decreases.

In the same paper, Ilmanen pointed out that the Genus 
Reduction Conjecture  ``includes the conjecture that every smooth self-shrinker of genus zero
is the plane, cylinder, or sphere.''
That conjecture about
genus~$0$ shrinkers was proved 
twenty years later by Brendle in his celebrated 2016 Annals of Math paper~\cite{brendle}.

Using Brendle's theorem, as well as a uniqueness of tangent flow theorem of Chodosh-Schulze~\cite{ChS}, Chodosh-Daniels-Holgate-Schulze proved, under some additional assumptions about the singularities, that the Genus Reduction Conjecture holds for innermost and outermost flows.
See~\cite{ChDHS}*{Remark~1.4(2)}.
By subsequent work
 of Chodosh-Choi-Schulze~\cite{ChChS} and
Bamler-Kleiner~\cite{bamler-kleiner}, we now know that the conjecture is true for innermost and outermost flows without any additional assumptions.  See Theorem~\ref{inner-outer-theorem} in the appendix.

 Expanders have been studied over the past 35 years.
In the graphical setting, Ecker and Huisken showed in \cite{EH} that entire graphical mean curvature
flows give rise, after rescaling, to self-expanding limits.  Angenent--Chopp--Ilmanen \cite{ACI}
studied rotationally symmetric expanders coming out of double cones and gave early
examples related to nonuniqueness and fattening. In \cite{ilmanen1},  Ilmanen proved, by variational
methods, the existence of an    area-minimizing surface (for the expander metric~\eqref{expander-metric}) asymptotic to any prescribed cone,
and Ding later refined this theory for smooth mean-convex cones~\cite{Ding2020} (see
 also~\cite{csch}.)  Bernstein and
Wang developed the analytic framework for asymptotically conical expanders,
including the Banach-manifold structure and compactness theory, and subsequently
introduced degree and min--max methods for producing new examples.  
They also proved topological uniqueness in the small-entropy regime.  
See~\cites{bw21,bw21b,bw1, bw2, bw3}.

 More recently,
Shao and Zou~\cite{sz-expanders} constructed the first examples of
positive-genus expanders, as well as the first example of a cone
admitting expanders of arbitrarily large genus asymptotic to it.
However, in the latter construction, the cone is rotationally symmetric about a line
and non-flat, and therefore there is no shrinker asymptotic to it.  (By a theorem of Lu Wang~\cite{Wang14}*{Corollary~14},
if a smooth complete properly embedded shrinker has a rotationally invariant conical end, then it must be rotationally invariant, and, thus, by Brendle's theorem~\cite{brendle}, it must be a plane.) 
Consequently, the expanders asymptotic to that cone
(unlike the expanders in
 Theorem~\ref{intro-theorem})
cannot not arise as the expanding side
of an eternal self-similar solution. 

\bigskip

The paper is organized as follows.  Section~\ref{preliminaries-section} fixes the notation and recalls the basic properties of the symmetry groups and of the expander metric.  In Sections~\ref{slices-section} and~\ref{big-small-section}, we introduce the two decompositions of the space of surfaces that will be used in the degree argument: the distinction between type~1 and type~2 surfaces, coming from the behavior of the symmetry slice, and the big/small dichotomy.  Section~\ref{sec:main} states the main existence theorem.  Sections~\ref{sec:degree} and~\ref{applying-section} develop the degree-theoretic framework and apply it to the four classes of surfaces. The key model case is treated in Sections~\ref{circular-section} and~\ref{uniqueness-section}, where we analyze the boundary consisting of three  horizontal circles close to the equator and prove 
 (for any specified genus)
existence, uniqueness, and strict stability of the maximally symmetric small, type 1 surface with that boundary. Section~\ref{big-section} uses the vanishing of a total degree to obtain the  corresponding big surfaces.  In Section~\ref{noescape-section} we let the radius tend to infinity and prove that the handles cannot escape, thereby obtaining complete asymptotically conical expanders with the prescribed genus.  Finally, Section~\ref{sec:ilmanen} applies the construction to disprove Ilmanen's genus-reduction conjecture, and the Appendix~\ref{sec:appendix} includes a proof that the genus-reduction conjecture
holds for innermost and outermost flows.

\section{Preliminaries} \label{preliminaries-section}

Throughout the paper we work in $\RR^3$.
Let $Z$ be the $z$-axis, 
$B(p,r)$ be the open ball of radius $r$ centered $p$,
and $\Ss^2$ be the unit $2$-sphere centered at the origin.

For $k\ge 1$,
 let $Q_k$ be the union of the $k$ horizontal lines (through the origin) on which the polar coordinate $\theta$ is an {\bf odd} multiple of $\pi/(2k)$.
Let $P_\alpha$ be the plane spanned by $\ee_3$ and
$(\cos\alpha,\sin\alpha,0)$.
Let $G_k$ be the group of isometries of $\RR^3$ generated by:
\begin{enumerate}
\item Rotations through $\pi$ about the lines in $Q_k$.
\item Reflections in the vertical planes $P_\theta$,
where $\theta$ is a multiple of $\pi/k$.
\end{enumerate}
In addition to those isometries, $G_k$ contains:
\begin{enumerate}
\setcounter{enumi}{2}
\item Rotations about $Z$ through integer multiples of $2\pi/k$.
\item Reflection in $\{z=0\}$ followed by rotation about $Z$ by
an odd multiple of $\pi/k$.
\end{enumerate}
Thus the group $G_k$ has $4k$ elements.
It  is the symmetry group of the Costa-Hoffman-Meeks surface~\cite{hoffman-meeks} whose intersection
with the plane $\{z=0\}$
is $Q_k$.

\begin{lemma}
\label{groups-lemma}
Suppose that $k\le n$.  The following are equivalent:
\begin{enumerate}
\item $Q_k\subseteq Q_n$.
\item $G_k\subseteq G_n$.
\item $n$ is an odd multiple of $k$.
\end{enumerate}
\end{lemma}

The proof is straightforward.

\begin{lemma}\label{symmetry of invariant curves}
Suppose that $k\ge 2$ and that $C$ is a $G_k$-invariant  simple closed curve in $\Ss^2\setminus Z$, each winding  once around $Z$.
Suppose also that $C\cap\{z=0\}$
is equal to $Q_k\cap \Ss^2$.  
 If $\sigma$ is an isometry of $\RR^3$ that maps $C$ to itself, then $\sigma\in G_k$.
\end{lemma}

\begin{proof}
The centers of mass of the
two components of $\Ss^2\setminus C$ lie 
in $Z\setminus\{0\}$.
Thus $\sigma$ must map $Z$ to itself, and,
therefore
 $\{z=0\}$ to itself.
It follows that either $\sigma$ is in $G_k$, or $\sigma$ is an element of $G_k$ followed by reflection in the plane $\{z=0\}$. 
In the latter case, $C$ would be invariant under that reflection, which is impossible. Thus $\sigma\in G_k$.
\end{proof}

\begin{lemma}
\label{gauss-bonnet-lemma}
Suppose $M$ is a smooth, connected expander in 
$\overline{B(0,R)}$.
Then
\begin{align*}
\frac12\int_M |A|^2\,d\Hh^2
&=
\frac18\int_M (p\cdot\nu)^2\,d\Hh^2p
+
\int_{\partial M}\kk\cdot \nn\,d\Hh^1
- 
2\pi\chi(M)
\\
&\le
\frac18 R^2
\, \Hh^2(M)
+
\int_{\partial M}\kk\cdot \nn\,d\Hh^1
-
2\pi\chi(M),
\end{align*}
where $A$ is the second fundamental form of $M$, $\nu$ is a unit normal to $M$, $\kk$ is the curvature vector of $\partial M$, $\nn$ is the unit normal vector to $\partial M$ that points out of $M$, and $\chi(M)$ is
the Euler characteristic of $M$.
\end{lemma}

\begin{proof}
Since $M$ is an expander, the mean curvature vector $H$
is $\frac12(\nu\cdot p)\nu$.

\[
 2K= |H|^2 - |A|^2
= \frac14 (p\cdot\nu)^2 - |A|^2.
\]
Now integrate and use the Gauss-Bonnet Theorem.
\end{proof}

For $\eps\ge 0$, consider the cone $C_\eps$ given by
\[
  z = \eps (x^2+y^2)^{1/2}.
\]
Ilmanen proved that there is a unique, smooth expander $E=E_\eps$ such $E$
is the graph of a smooth, Lipschitz function from $z=f_\eps(x,y)$ and such that $E_\eps$ is asymptotic to $C_\eps$ at infinity.
By uniqueness, $E_\eps$
is rotationally invariant
about $Z$.
Ilmanen's proof shows that
\begin{equation}
\label{above-the-cone}
\text{if $\eps>0$, then $E_\eps$
lies above the cone $C_\eps$}.
\end{equation}
Also, by the strong maximum principle, if $\delta<\eps$, then $E_\delta$ lies below $E_\eps$.
Indeed, writing \(E_\eta\) as the graph \(z=f_\eta(x,y)\), the functions
\(f_\eta\) solve the graphical expander equation
\[ \operatorname{div}\left(\frac{Df_\eta}{\sqrt{1+|Df_\eta|^2}}\right)
 =
 \frac{f_\eta-x\cdot Df_\eta}{2\sqrt{1+|Df_\eta|^2}}.
\]
If \(\delta<\eps\), then, since \(f_\delta\) and \(f_\eps\) are asymptotic to
\(\delta r\) and \(\eps r\), respectively, we have
\(f_\delta-f_\eps\to -\infty\) as \(r\to\infty\).  Thus, if
\(f_\delta-f_\eps\) were positive somewhere, it would attain a positive
interior maximum.  At such a point the gradients agree and
\(D^2f_\delta\leq D^2f_\eps\), so the left-hand side of the expander
equation for \(f_\delta\) is no larger than the left-hand side for
\(f_\eps\), while the right-hand side is strictly larger.  This contradiction
shows that \(f_\delta\leq f_\eps\), and the strong maximum principle gives
\(f_\delta<f_\eps\) unless the two expanders coincide.  Since their
asymptotic cones are different, they cannot coincide.

Consequently, the $E_\eps$ ($\eps\in \RR$) form a foliation $\Ee$ of $\RR^3$.

As mentioned in the introduction, Ding later refined this construction for smooth mean-convex cones~\cite{Ding2020}.

\begin{definition}
\label{zeta-definition}
Let $\zeta:\RR^3\to\RR$
be the function such that,
for each $s\in\RR$, 
\[
  \{\zeta=s\}
\]
is the leaf of the foliation $\Ee$ that contains the point $(0,0,s)$.
\end{definition}

\begin{theorem}[Monotonicity Theorem]
\label{monotonicity-theorem}
If $E$ is a compact, embedded expander, then
\begin{equation}
\label{cone-comparison}
  \Hh^2(E)
  \le
  \frac12 \int_{\partial E} |p|\,d\Hh^1(p).
\end{equation}
If $\partial E$ lies outside of $B(0,R)$, then 
\begin{equation}
\label{monotonicity}
   \frac1{\pi r^2}\Hh^2(E\cap B(0,r))
\end{equation}
is an increasing function of 
 $r\in (0,R]$.
\end{theorem}

\begin{proof}
Consider the vectorfield $X(p):=p$.
At each point of $E$, $\Div_EX= 2$ and
\[
   H(p)\cdot X(p)= \frac12p^\perp\cdot p \ge 0
\]
(by~\eqref{expander-equation}), so,
letting $\nu$ be the unit normal to $\partial E$ that points out of $E$,
\begin{align*}
2\area(E)
&= \int_E\Div_E X\,d\Hh^2
\\
&= -\int_E H\cdot X\,d\Hh^2
    +
   \int_{\partial E}X\cdot \nu\,d\Hh^1(p)
   \\
&\le
  \int_{\partial E}|p|\,d\Hh^1(p).
\end{align*}
The monotonicity of~\eqref{monotonicity} follows
from~\eqref{cone-comparison}
by the coarea formula
(exactly as in the proof of monotonicity for surfaces that are minimal with respect to the Euclidean metric.)
\end{proof}

\begin{remark}
Theorem~\ref{monotonicity-theorem}
generalizes (with the same proof)
to the analogous theorem for $m$-dimensional integral varifold expanders in $\RR^N$.
\end{remark}

\section{Behavior of Slices} \label{slices-section}

\newcommand{\pup}{p_{\rm upper}}
\newcommand{\pmid}{p_{\rm middle}}
\newcommand{\plow}{p_\textnormal{\,lower}}
Suppose $M$ is a smooth embedded surface in $\overline{B(0,R)}$ 
that is invariant under reflection in the plane $\{y=0\}$.
Suppose that $\partial M$ consists  of three simple closed curves in $(\partial B(0,R))\setminus Z$, each of which winds once around $Z$.  
Then $\partial M$ intersects the halfplane $\{y=0,\,x>0\}$ in three points, which we denote by  $\pup$, $\pmid$,
and $\plow$, where
$z(\pup)>z(\pmid)>z(\plow)$.

By the reflectional symmetry, $\Gamma:=M\cap \{y=0\}$
is a collection of disjoint, smooth curves.

\begin{definition}
\label{type-definition}
We say that $M$ has 
{\bf type 1}
if $\Gamma$ contains no closed curve and if one component of $\Gamma$ joins $\pmid$ to $\pup$.
We say that $M$ has
 {\bf type 2}
if $\Gamma$ contains no closed curve and if one component of $\Gamma$ joints $\pmid$ to $\plow$.
We say that a complete, properly embedded surface $M$ in $\RR^3$ is type 1 (or type 2) if $M\cap B(0,R)$ is type 1 (or type 2)
for all sufficiently large $R$.
\end{definition}

\begin{figure}[htpb]
    \centering
    \includegraphics[width=0.45\linewidth]{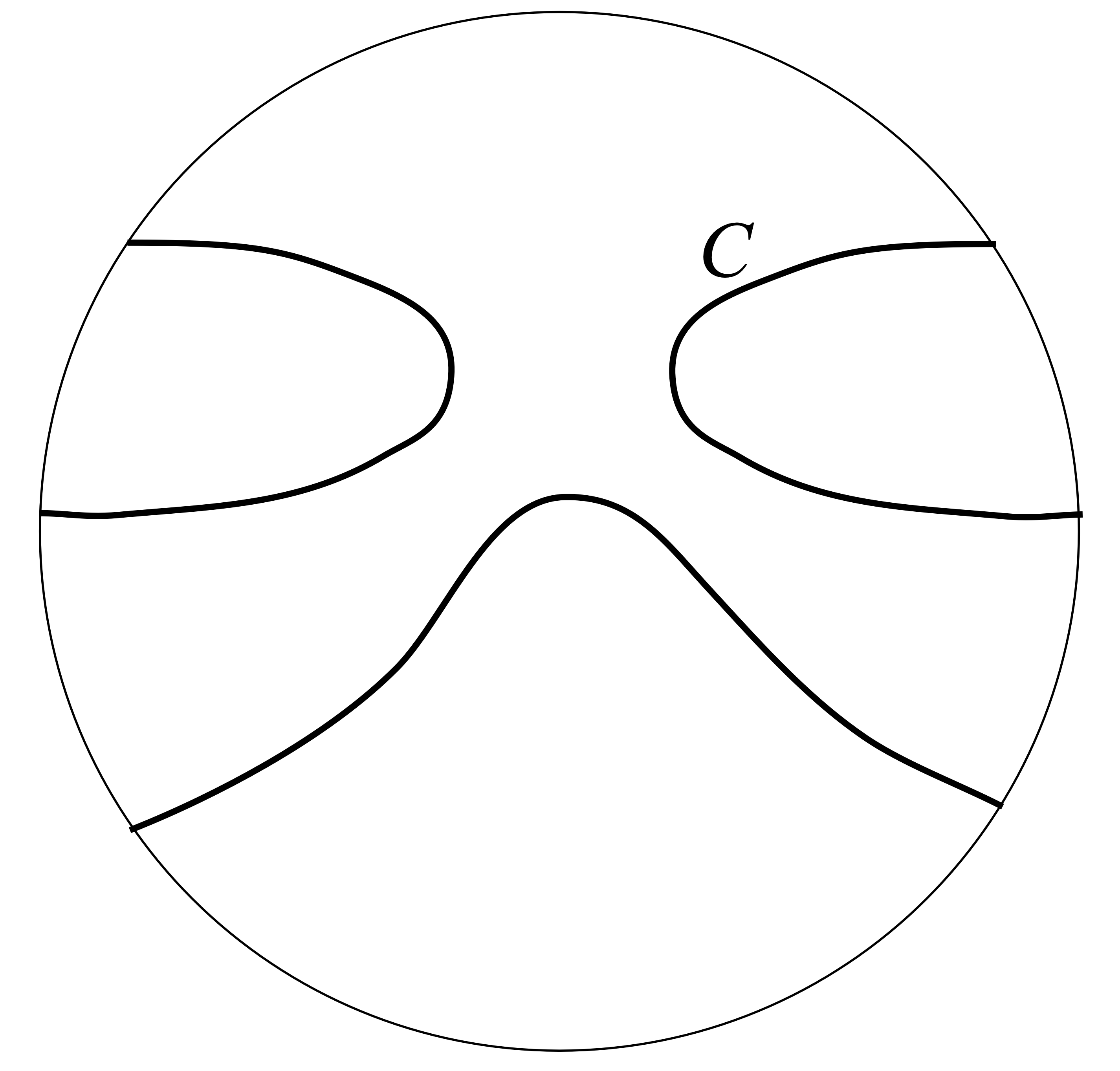} \includegraphics[width=0.45\linewidth]{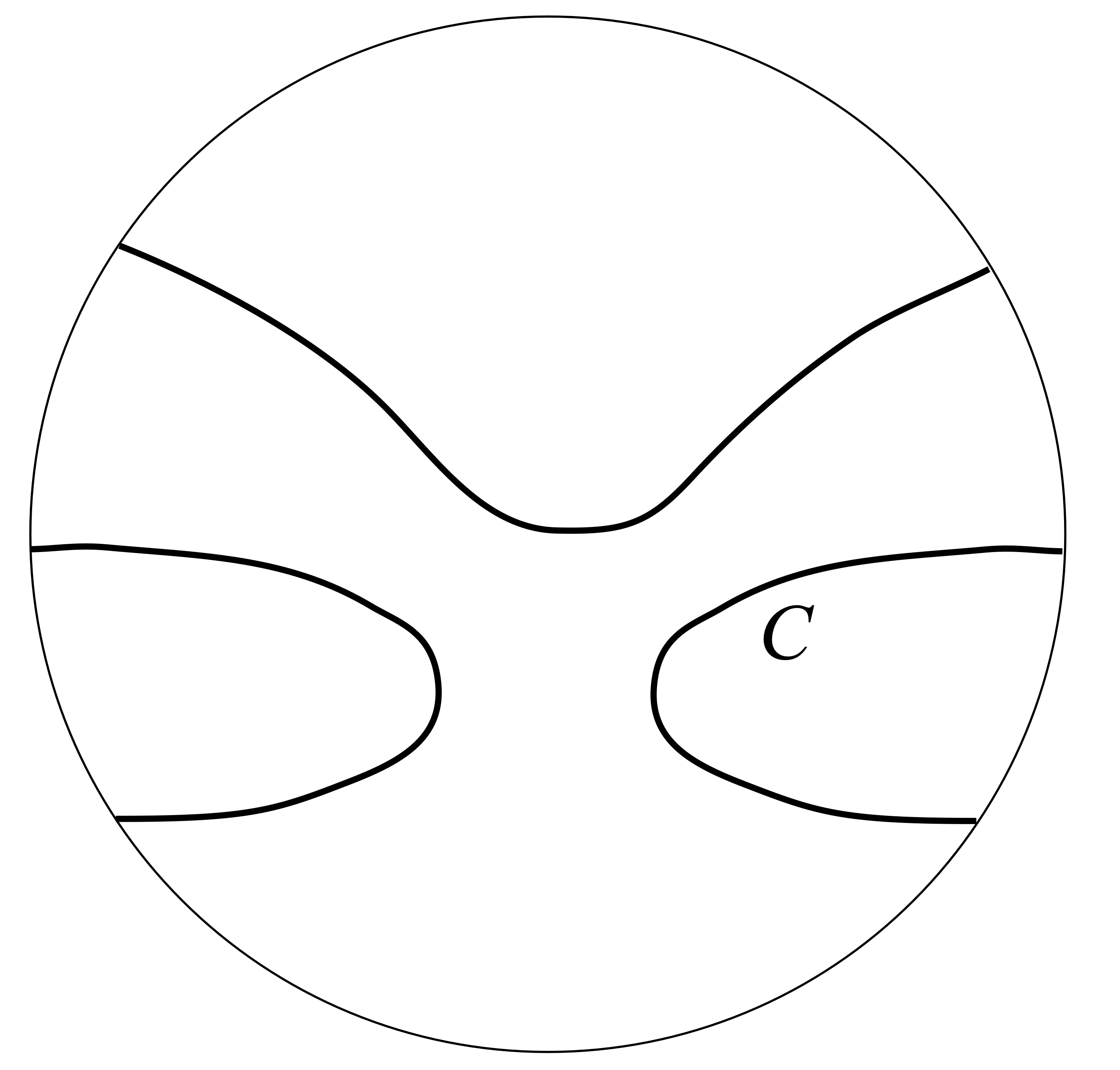}
    \includegraphics[width=0.45\linewidth]{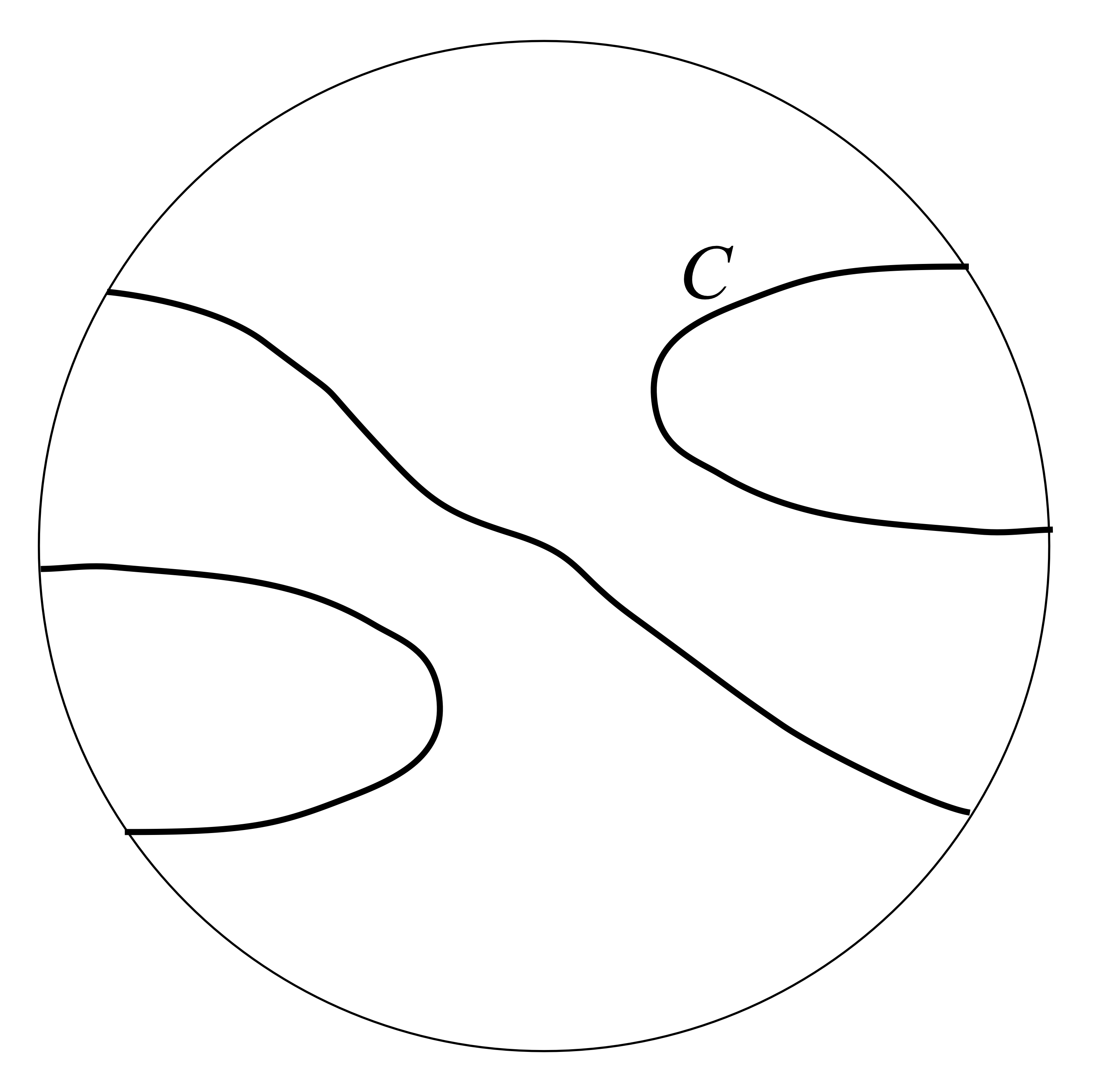}
\includegraphics[width=0.45\linewidth]{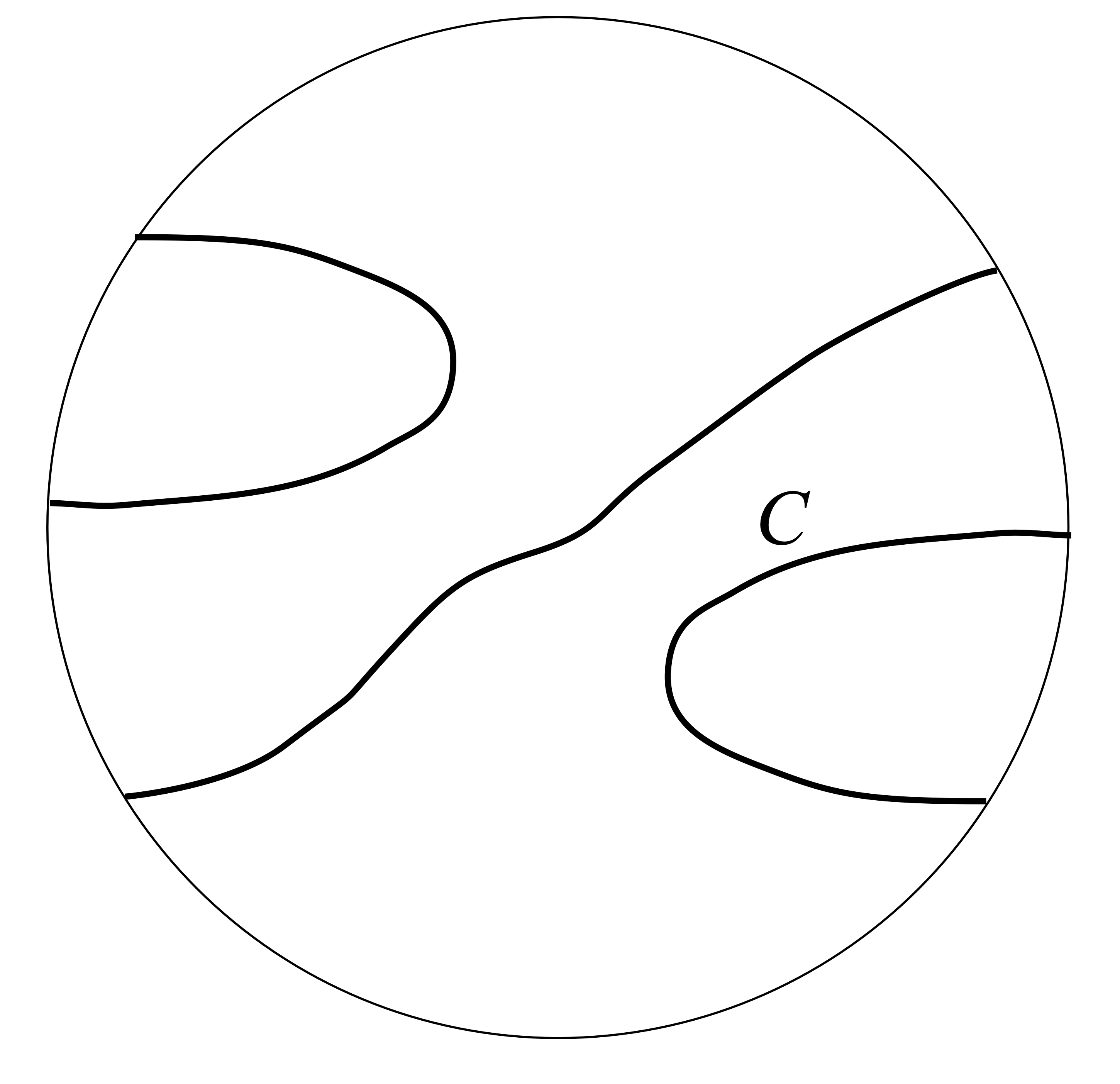}

    \caption{The $\{y=0\}$ slice of a type 1 surface (left) and a type 2 surface (right).  The component $C$ joints $\pmid$ to $\pup$ (left) or $\plow$ (right).}
    \label{fig:C}
\end{figure}

\begin{proposition} \label{noncongruence-proposition}
Suppose that $k\ge 2$,
 that $M_1$ and $M_2$
are $G_k$-invariant surfaces embedded in $\overline{B(0,R)}$, 
that $M_1$ and $M_2$ 
have the same boundary $C\subset \partial B(0,R)$, and that
\[
 \Gamma\cap\{z=0\} 
 =
 Q_k\cap \partial B(0,R).
\] 
If $M_1$ and $M_2$ are are of different types, they are not congruent.
\end{proposition}

\begin{proof}
Suppose $M_1=\sigma M_2$
for some isometry $\sigma$ of $\RR^3$.
Then $\sigma(\Gamma)=\Gamma$,
so, by Lemma~\ref{symmetry of invariant curves}, $\sigma \in G_k$, so $\sigma M_1=M_1$ and therefore $M_1=M_2$, a contradiction.
\end{proof}

\section{Big Surfaces and Small Surfaces}
\label{big-small-section}

Fix a smooth, nonnegative, $O(3)$-invariant function 
  $\phi:\RR^3\to \RR$
 such that $\phi$ is supported in $B(0,1/2)$
 and such that
\[
  \int_{\{z=0\}}\phi\,d\Hh^2 = 1.
\]

\begin{definition} \label{big-small-definition}
If $M$ is a surface in $\RR^3$, we say that
$M$ is {\bf big}
if 
\[
 \int_M \phi\,d\Hh^2 \ge \frac32
\]
and {\bf small} if
\[
 \int_M \phi\,d\Hh^2 < \frac32.
\]
\end{definition}

Consider the space 
  $\mathcal{E}$ of all smooth embedded expanders in $\RR^3$
whose boundary (if nonempty) lies outside of the unit ball $B(0,1)$.  It is not hard to show that the 
set of small expanders in $\Ee$ is a open, but not closed, subset of $\Ee$.  

In this section, we will identify a smaller class of expanders in which ``big" and ``small'' define subsets that are both open and closed.

\begin{theorem} \label{th:eta}
Suppose $M_i$ is a sequence of compact expanders such that 
\[
 \partial M_i \subset
 \RR^3\setminus B(0,1) 
\] 
and such that
\[
   \eta_i:=\max_{\partial M_i}|\zeta(\cdot)|
   \to 0,
\]
where $\zeta(\cdot)$ is as in 
Definition~\ref{zeta-definition}.
Then, after passing to a subsequence,
\[
a_i:=\int_{M_i}\phi\,d\Hh^2
\]
converges to an integer or to $+\infty$.
\end{theorem}

\begin{proof}
By passing to a subsequence, we can assume
that the $a_i$ converge
to a limit $a$ in $[0,+\infty]$.
If $a=\infty$, we are done.
Thus assume $a<\infty$.

By the maximum principle,
\[
\max_{M_i}|\zeta(\cdot)|
  =
\max_{\partial M_i}|\zeta(\cdot)| 
\to 
0.
\]

Let $M_i'=M_i\cap B(0,1)$. If $H(M_i',p)$ is the mean curvature vector of $M_i$ at $p$, then
\[
|H(M_i',p)|
=
|p^\perp/2| 
\le
\frac12.
\]
Let $\pazocal{Q}$ be the area blow up set, i.e, the set of points
$q\in B(0,1)$ such that
\[ 
  \limsup_{i\to\infty}\area(M_i\cap B(q,r)) = \infty.
\]
for every $r>0$.
Note that
\[
  \pazocal{Q} \subset D_1\setminus B(0,1/2).
\]
Thus (by Theorem~7.3 in \cite{white-controlling}) $\pazocal{Q}=\varnothing$.

Consequently, after passing to a subsequence,  the
$M_i$ converge (as varifolds)
to an integral varifold $V$
with mean curvature $\le \frac12$
and with no generalized boundary in the open ball $B(0,1)$.  By the
constancy theorem, $V$ is the disk $D$ with some nonnegative integer multiplicity.
\end{proof}

\begin{theorem}
\label{epsilon-theorem}
There is an $\eps>0$ with the following property.
Let $\Cc=\Cc_\eps$ be the collection
of compact expanders $M$ such that
\[
   \partial M \subset \RR^3\setminus B(0,1)
\]
and such that $\partial M$
is contained in the region
\[
   |z|
   \le
   \eps (x^2+y^2)^{1/2}.
\]
Then 
\[
  \int_M \phi\,d\Hh^2 <\frac43
\]
or
\[
  \int_M \phi\,d\Hh^2 
  >
  \frac53.
\]
\end{theorem}

\begin{proof}
Suppose not. Then there exist
examples $M_i$ and $\eps_i$
with $\eps_i\to 0$ such that
\[
  \frac43
  \le
  \int_{M_i}\phi\,d\Hh^2
  \le
  \frac53.
\]
By passing to a subsequence, we can assume that the integrals converge
\begin{equation} \label{eq:integralsM_i}
    \lim_i\int_{M_i}\phi\,d\Hh^m
 \in 
 [4/3,5/3].
\end{equation}

For each $i$, let $\eta_i>0$
be such that the expander
$\{\zeta=\eta_i\}$ is asymptotic to the cone
\begin{equation}
\label{the-cone}
  z=\eps_i(x^2+y^2)^{1/2}
\end{equation}
at infinity.
 The expander $\{\zeta=\eta_i\}$ lies above the cone~\eqref{the-cone} by~\eqref{above-the-cone}.

Consequently, if
\[
 |z|\le \eps_i(x^2+y^2)^{1/2},
\]
then 
\[
  |\zeta(x,y,z)| \le \eta_i.
\]
Therefore,
\[
   \max_{\partial M_i}|\zeta(\cdot)| \le \eta_i \to 0.
\]
Thus by Theorem~\ref{th:eta}, the limit of the integrals in \eqref{eq:integralsM_i} is either a nonnegative integer or $\infty$,
a contradiction.
\end{proof}

\begin{corollary}
\label{epsilon-corollary}
Any limit of big surfaces in $\Cc_\eps$ is big, and any limit
of small surfaces in $\Cc_\eps$
is small, where
big and small are as
in Definition~\ref{big-small-definition}.
\end{corollary}

\begin{remark}
Some constants that occur later
in the paper are allowed to depend on radius, genus, etc.
But the $\eps$ 
 in Theorem~\ref{epsilon-theorem} really is a constant: it does not depend on anything.
  This is crucial: it means that if $k$ is large enough, then the $\eps_k$ in Theorem~\ref{shrinker-theorem} is $<\eps$.
\end{remark}

\section{The Main Theorem} \label{sec:main}

\begin{theorem} 
\label{main-theorem}
Let $k$ be a positive integer.
Suppose $C$ is a $G_k$-invariant cone in the region
\[
   z\le \eps \, (x^2+y^2)^{1/2},
\]
where $\eps$ is as in Theorem~\ref{epsilon-theorem},
and  that 
 $C\cap \partial B(0,1)$
is the union of $3$ disjoint, smoothly embedded,  simple closed curves, each
of which winds once around $Z$.

Suppose $g$ is of the form
$p^nk-1$, where $p$ is an odd prime
and $n\ge 0$.  Then for each $R\in [1,\infty)$,  $C \cap \partial B(0,R)$
bounds at least one
connected, embedded, $G_k$-invariant, genus-$g$
expander of each of the following four varieties:
\begin{enumerate}
\item big and type 1.
\item small and type 1.
\item big and type 2.
\item small and type 2.
\end{enumerate}
Furthermore, for each of those four varieties, there is a complete, connected, properly embedded, genus-$g$ expander 
that is smoothly asymptotic to $C$ at infinity. 
\end{theorem}

Clearly no big surface is congruent to any small surface by an element of $O(3)$. 
Furthermore, 
if $G_k$
is the entire symmetry group of $C$, then (trivially) no two 
of the four expanders in Theorem~\ref{main-theorem} are congruent to each other. 
By
Proposition~\ref{noncongruence-proposition}, that is the case for the cones 
(coming from Theorem~\ref{shrinker-theorem}) that we use for counterexamples to the Genus Reduction Conjecture. 

We will prove Theorem~\ref{main-theorem} for finite $R$ in the next few sections using degree theory.
Specifically, existence of a small expanders of types 1 and 2 is proved in Corollary~\ref{corollary-big-small},
and existence of large expanders of types 1 and 2 is proved in 
Corollary~\ref{big-corollary}.

The existence of the complete
expanders
follows by taking subsequential limits as     $R\to\infty$;
see Theorem~\ref{R-to-infinity-theorem}.

\section{Degree Theory} \label{sec:degree}

Let $N$ be a compact, smooth, $3$-dimensional
Riemannian
manifold-with-boundary
such that 
\begin{enumerate}
\item At each point of the boundary, the mean curvature vector is nonzero and points into $N$.
\item $N$ contains no closed minimal surface.
\end{enumerate}
For us, $N$ will be a closed ball $\overline{B(0,R)}$ (with $R\ge 1$) in $\RR^3$ with the expander metric~\eqref{expander-metric}.

Let $G$ be a finite group of isometries of 
 $N$.

Let $\Bb$ be the set of all $G$-invariant $C$ such that $C$ is the
union of $3$ disjoint,  smoothly embedded,
closed curves in $\partial N$.
Let $\MM$ be the space
of all $M$ such that 
$M$ is a $G$-invariant, 
compact, connected, smoothly embedded genus-$g$ expander in $\RR^3$ with $3$ boundary components.

If $S\subset \Bb$,
we let $\Mm(S)$ denote the set
of all $M\in \Mm$ such that $\partial M\in S$.
More generally, if $S\subset \Bb$
and if $K\subset \MM$, 
we let $K(S)$ denote the set
of all $M\in K$ such that $\partial M\in S$.
In case of $S$ consists of a single element $C$, we will sometimes write $C$ in place of $\{C\}$ 
and $K(C)$ in place
of $K(\{C\})$ to reduce 
 notational clutter.

Note that if $S\subset \Bb$
and if $K\subset \MM(S)$, then we have map
\begin{equation}
\label{the-map}
M\in K  \mapsto \partial M\in S.
\tag{*}
\end{equation}
The following definition and
 theorem give a natural condition on the pair $(S,K)$
guaranteeing that the map~\eqref{the-map} has a
mapping degree, which we will denote by $\deg(S,K)$.

\begin{definition}
\label{Ff-definition}
Let $\Ff$ be the set of all pairs $(S,K)$ such that
$S$ is a path-connected
set in $\Bb$
and such that $K$ is a relatively open-and-closed subset of $\MM(S)$.
\end{definition}

Note that if $(S,K)$ is in $\Ff$,
and if $C\in S$, 
then 
\[
   (C, K(C))
\]
is also an element of $\Ff$.

We will use the following special case of~\cite{white-degree}:

\begin{theorem}
\label{basic-theorem}
There is a unique
map
\[
    \deg: \Ff \to \ZZ
\]
with the following properties. 
If $(S,K)\in \Ff$ and if
 $C\in S$, then
\begin{enumerate}
\item
\label{pointwise-item}
$
 \deg(S,K)
 =
 \deg(C, K(C)).
$
\item
\label{sum-item}
If, in addition,
 each surface in $K(C)$ has $G$-nullity $0$, then 
\[
 \deg(S,K)
 =
 \sum_{M \in K(C)}
 (-1)^{G\mh\Index(M)}.
\]
\item 
\label{empty-item}
If $K(C)=\varnothing$,
then $\deg(S,K)=0$.
\item
\label{additivity-item}
If $K$ is the disjoint union of two relatively open-and-closed subsets $K_1$ and $K_2$, then
\[
  d(S,K)=d(S,K_1) + d(S,K_2).
\]
\end{enumerate}
\end{theorem}

We remark that 
 Property~\ref{empty-item}
 is a special case of 
 Property~\eqref{sum-item}.

We will also use the following theorem from~\cite{white-degree}:

\begin{theorem}\label{mod-p-theorem}
In Theorem~\ref{basic-theorem}, suppose that $G$ is a subgroup of 
a finite group $H$ of isometries $N$, and that 
 $[H:G] = p^n$
for some odd prime number $p$ and some integer $n\ge 0$.
Suppose that $(S,K)\in \Ff$,
that $C\in S$, and
that exactly one surface $M_0$ in $K(C)$ is $H$-invariant,  and that the $H$-nullity of $M_0$  is $0$.  Then
 $\deg(S,K)$ is congruent mod $p$ to $1$ or to $-1$.
In particular, $\deg(S,K)\ne 0$.
\end{theorem}

(Recall that $[H:G]$ is the number of elements of $H$ divided by the number of elements of $G$.)

\begin{proof}[Idea of proof]
For simplicity, suppose $n=1$. 
Consider the case when
every surface in $K(C)$ has $G$-nullity $0$, and thus
\begin{equation}\label{count}
  \deg(S,K)
  =
  \sum_{M\in K(C)}
  (-1)^{G\mh\Index (h,M)}.
\end{equation}
Each surface in $K(C)$
that is not $H$-invariant has $p-1$ other surfaces related to it by the action of $H$, and those $p$ surfaces together contribute
$p$ or $-p$ to the
 count~\eqref{count}.
In particular, they contribute $0$ mod $p$.
The result follows immediately. 

The general result (when $K(C)$ may contain surfaces
with nonzero $G$-nullity)
follows by an argument in which one perturbs the metric slightly to get
an $H$-invariant metric $h'$ that is bumpy for $C$, i.e.,
for which $C$ bounds no smooth $h'$-minimal surfaces with nonzero nullity.
See~\cite{white-degree} for details.
\end{proof}

\section{Applying the Degree Theory} \label{applying-section}

In this section, we fix a $k\ge 1$ and a genus $g\ge 0$.
Everything below will depend on $k$ and $g$, but
we will not show the dependence in the notation.  We also fix
an $R\ge 1$, and
we let $N=\overline{B(0,R)}$
with the expander metric.

To apply the theory described in Section~\ref{sec:degree}, we need to specify the group $G$ of diffeomorphisms,
the class $S\subset \Bb$ of boundaries, and the class $K\subset \MM(S)$
of surfaces.

We let $G$ be the group $G_k$.

\begin{definition}\label{Sc-definition}
Let $\Sc$ be the class of all curves $C$ 
such that
\begin{enumerate}
\item 
$C$ is contained in $(\partial B(0,R))\setminus Z$ for some $R\ge 1$, and in
\[
  \{(x,y,z): |z|\le \eps(x^2+y^2)^{1/2}\},
\]
where $\eps$ is as in Theorem~\ref{epsilon-theorem}.
\item $C$ is a $G_k$-invariant union of three disjoint, smoothly embedded curves, each of which winds
once around $Z$.
\end{enumerate}
\end{definition}

\begin{lemma}
\label{connectedness-lemma}
The set $\Sc$ 
is connected.
\end{lemma}

\begin{proof}
For each $C\in \Sc$,
let $\tilde C$ be the image of $C$ under radial projection
to the cylinder $x^2+y^2=1$.
We prove the lemma by proving that
\[
\tilde{\Sc}
:=
\{\tilde C: C\in \Sc\}
\]
is connected.

Let $\tilde C(t)$ be the result of letting a curve $\tilde C\in \tilde{\Sc}$ 
flow for time $t$ be the curve-shortening flow in the cylinder $x^2+y^2=1$.
The flow preserves symmetries and smooth embeddedness.  By the maximum principle,
\[
 \max_{\tilde C(t)} |z|
\]
is a decreasing function of~$t$.  Thus $\tilde C(t)\in \tilde{\Sc}$ for all $t\ge 0$.  By Grayson's theorem~\cite{grayson}, $\tilde C(t)$ converges smoothly as $t\to\infty$ to three horizontal circles.
Connectedness (indeed, contractibility) of 
  $\tilde{\Sc}$ and therefore of $\Sc$
  follows immediately.
\end{proof}

Recall that $\MM$ is the set of all  $G_k$-invariant, smoothly embedded,  genus-$g$ expanders in $N$
with $3$ boundary components.

We have four choices for the set $K$:

\begin{definition} \label{def:K's}
We let 
$K\bone$ be the set of
 $M\in \MM(\Sc)$ such that $M$ is big and
  type~1. (See Definitions~\ref{type-definition} and \ref{big-small-definition}.) We define
  $K\btwo$, $K\sone$,
  and $K\stwo$ 
  analogously.
\end{definition}

\begin{lemma}
\label{closed-conditions-lemma}
Suppose that
$M_i\in \MM(\Sc)$
  converges
smoothly  to a  
limit surface
 $M\in M(\Sc)$.
Then
\begin{enumerate}
\item 
\label{ass-1}
$M$ is of type 1 if and only $M_i$ is of type 1 for all sufficiently large $i$.
\item 
\label{ass-2}
$M$ is of type 2 if and only $M_i$ is of type 2
for all sufficiently large $i$.
\item
\label{ass-3}
$M$ is big if and only if $M_i$ is big for all sufficiently large $i$.
\item
\label{ass-4}
$M$ is small if and only if $M_i$ is small for all sufficiently large $i$.
\end{enumerate}
\end{lemma}

\begin{proof}
Assertions~\eqref{ass-1} 
 and~\eqref{ass-2} are trivial consequences of the smooth convergence of $M_i$ to $M$.
Assertions~\eqref{ass-3}
 and~\eqref{ass-4}
are special cases of
Corollary~\ref{epsilon-corollary}.
\end{proof}

The following is an immediate consequence of Lemmas~\ref{connectedness-lemma} and~\ref{closed-conditions-lemma}.

\begin{theorem}\label{th:deg(S,K)}  
Let $K$ be one of the four sets
specified in Definition~\ref{def:K's}.
Then $(\Sc,K)$ belongs to the $\Ff$
of Theorem~\ref{basic-theorem}, and, consequently 
(by that Theorem),
$(\Sc,K)$ has a degree $\deg(\Sc,K)$.
\end{theorem}

We denote  
$\deg(\Sc,K)$
by $\dbone$, $\dbtwo$,
$\dsone$, or by $\dstwo$
according to whether $K$ is $K\bone$, $K\btwo$, $K\sone$, or $K\stwo$.

By symmetry,
\begin{equation} \label{eq:symmetry}
\begin{aligned}
\dbone &= \dbtwo, 
 \\
\dsone &= \dstwo.
\end{aligned}
\end{equation}

(If that is not clear, note
that reflection in the
 plane $\{z=0\}$ leaves $\Sc$ unchanged, and gives bijections from $K\bone$ to $K\btwo$ and from 
$K\sone$ to $K\stwo$.)

At the beginning of this section, we fixed a $k$ and a genus $g$.
Of course the values of the degrees~\eqref{eq:symmetry} may depend on $k$ and on $g$.

\begin{theorem} 
\label{small-theorem}
Suppose the genus $g$ is of the form $p^nk-1$ for some odd prime number $p$ and some $n\ge 0$.  
Then $\dsone$ is congruent to $1$ or to  $-1$ mod $p$.
In particular, $\dsone\ne 0$.
\end{theorem}

\begin{proof}[Proof 
of Theorem~\ref{small-theorem}]
We want to use Theorem~\ref{mod-p-theorem}. We will show that all the hypotheses of that theorem are satisfied. By Lemma~\ref{groups-lemma},
 $G:=G_k$ is a subgroup of $H:=G_{p^nk}$
and 
\[
 [H:G] 
 =
 \frac{o(H)}{o(G)}
 =
 \frac{4p^nk}{4k}
 = p^n.
\]

For $0<s<1$, let $C(s)$ the
union of the circles of intersection of
  $\partial B(0,R)$
with the translators
 $\{\zeta=s\}$, 
 $\{\zeta=0\}$, and
 $\{\zeta=-s\}$.
 (See Definition~\ref{zeta-definition}).
For all sufficiently small $s$,
$C(s)\in \Sc$.

In Sections ~\ref{circular-section}
and~\ref{uniqueness-section},
we will show (see Theorem~\ref{strict-theorem}),
for all sufficiently small $s>0$,
that 
there is exactly one small, type 1, $H$-invariant expander $M$ with genus $g=p^nk$, and with boundary $C(s)$, and 
that its $H$-nullity is $0$.
(We are invoking Theorem~\ref{strict-theorem} with the $k$ there replaced by $p^nk$.)
Thus, by definition of $d\sone$ and by Theorem~\ref{mod-p-theorem}, 
\[
d\sone
=
\deg(\Sc,K\sone)
\equiv
\pm 1 
\mod p.
\]
In particular, $\dsone$ is nonzero.
\end{proof}

\newcommand{\cS}{\mathcal{S}}
\newcommand{\cK}{\mathcal{K}}
\newcommand{\type}{\mathrm{type}}
\newcommand{\smalltype}{\mathrm{small}}
\newcommand{\degop}{\operatorname{deg}}

\begin{corollary} \label{corollary-big-small}
Suppose $C$ is a curve
 in $\Sc$ 
(where $\Sc$ is as in Definition~\ref{Sc-definition}).
Then $C$ bounds at least
one small, $G_k$-invariant
expander of type 1
and at least one small, $G_k$-invariant expander
of type 2.
\end{corollary}

\begin{proof}
By Theorem~\ref{small-theorem} and equation \eqref{eq:symmetry}, $\dsone=\dstwo\neq 0$.
The corollary then follows
immediately
from Theorem~\ref{basic-theorem}\eqref{empty-item}.
\end{proof}

\section{Jacobi Fields} \label{sec:jacobi}

Here we recall some standard facts
about Jacobi fields, and deduce 
some corollaries  that will be used in later sections.
Let $\gamma$ be a smooth Riemannian metric on $\RR^3$,
and let $M\subset \RR^3$ 
 be an oriented $\gamma$-minimal surface.  Let $\nu(p)$ be the Euclidean unit normal to $M$ at $p$.  
A {\bf Jacobi-field function} on $M$ is a function $u:M\to \RR$ such that 
\[
   p\in M \mapsto u(p)\nu(p)
\]
is a $\gamma$-Jacobi field on $M$ (i.e., a Jacobi field with respect to the metric $\gamma$.)

The cases of interest for us are the Euclidean metric and the expander metric.
For the Euclidean metric, $u$ is a Jacobi-field function if and only if
\begin{equation}
\label{jacobi-equation}
  \Delta u - 2Ku = 0,
\end{equation}
For the expander metric~\eqref{expander-metric}, $u$ is a Jacobi-field function if and only if 
\begin{equation}
\label{expander-jacobi-equation}
 \Delta u + \frac12 \langle x,\nabla_M u\rangle
+ \left(|A_M|^2-\frac12\right)u = 0.
\end{equation}
In~\eqref{jacobi-equation} 
 and~\eqref{expander-jacobi-equation}, all quantities are with respect to the Euclidean metric.

\begin{lemma}
\label{maximum-principle}
Suppose that $u$ is a solution of~\eqref{jacobi-equation} on a connected open subset of a plane, or that $u$ is a solution
 of~\eqref{expander-jacobi-equation}
on a connected open subset of the plane $\{z=0\}$.
If $|u(\cdot)|$ attains a maximum, then it is constant.
\end{lemma}

Lemma~\ref{maximum-principle} follows immediately from the strong maximum principle.

\begin{lemma}\label{pair-lemma}
Let $\gamma_i$ be a sequence of smooth Riemannian metrics on $\RR^3$ that converges smoothly to a metric $\gamma$.
Suppose $M_i$ and $N_i$ be sequences of embedded, oriented, $\gamma_i$-minimal surfaces in $\RR^3$ such that, for some open subset $U$ of $\RR^3$,
\[
  (\partial M_i)\cap U = (\partial N_i)\cap U
\]
and such that $M_i\cap U$ and $N_i\cap U$
converge smoothly and with multiplicity
one to a connected, $\gamma$-minimal surface $\Sigma$.

Suppose that
\[
   p\in M_i\mapsto \dist(p,N_i)
\]
attains its maximum (which is $>0$) at a point $p_i$,
where $\dist$ is distance with respect to the Euclidean metric.
Suppose also that $p_i$ converges to a point $p\in \Sigma$.

Then there is an exhaustion $U_i$ of $U$
and there are smooth functions 
\[
  f_i:U_i\cap M_i\to \RR 
\]
such that
\[
   p + f_i(p)\nu(M_i,p) \in N_i
\]
for all $p\in M_i\cap U_i$ 
and such that 
\[
  \lim_i \sup(|f_i(\cdot)|) = 0.
\]
Furthermore, after passing to a subsequence, the functions 
\[
   f_i(\cdot)/f_i(p_i)
\]
converge smoothly (on compact subsets of $U$) to a Jacobi-field function 
\[
   u: \Sigma \to \RR
\]
such that
\[
   \max |u(\cdot)| = u(p)=1.
\]
\end{lemma}

\begin{corollary}\label{pair-corollary}
In Lemma~\ref{pair-lemma},
suppose that $\gamma$ is the Euclidean metric and that
that $\Sigma$ is planar,
or that $\gamma$ is the expander metric and that $\Sigma$ is contained in the plane $\{z=0\}$.
Suppose also that $q_i\in M_i\cap N_i$ converges to a point $q$. 
Then $q\notin U$.
In particular, 
\[
   (\partial \Sigma)\cap U = \varnothing.
\]
\end{corollary}

\begin{proof}
By Lemma~\ref{maximum-principle},
$
   u\equiv 1.
$
If $q_i\in M_i\cap N_i$, then $f_i(q_i)=0$.
Thus if $q_i$ converged to a point 
$q\in \Sigma$, then $u(q)=0$, contradicting $u \equiv 1$.
\end{proof}

\begin{lemma}\label{jacobi-lemma}
Let $M_i$, $\gamma_i$, $\Sigma$, $\gamma$, and $U$ be as in Lemma \ref{pair-lemma}.
Suppose that $u_i(\cdot)$
is a Jacobi-field function
on $M_i$ that vanishes
on $(\partial M_i)\cap U$.
Suppose also that 
\[
  \max_{M_i\cap U} |u_i(\cdot)|
  = u_i(p_i) = 1
\]
for some $p_i$ that converge to a limit $p\in \Sigma$.
Then the $u_i(\cdot)$ converge smoothly 
on compact subsets of $U$ to 
a Jacobi-field function
\[
  u:\Sigma\to \RR
\]
such that 
\[
  \max |u(\cdot)| = u(p) = 1.
\]
\end{lemma}

\begin{corollary}\label{jacobi-corollary}
In 
Lemma~\ref{jacobi-lemma}, suppose that $\gamma$ is the Euclidean metric and that $\Sigma$ is planar,
or that $\gamma$ is the expander metric
and that $\Sigma$ is contained in the plane   $\{z=0\}$.
Then $u\equiv 1$.
If $u_i(q_i)=0$ and $q_i$ converges to a point $q$, then $q\notin U$.
In particular, 
\[
  U\cap\partial \Sigma=\varnothing.
\]
\end{corollary}

\begin{proof}
If $q\in U$, then $u(q)=\lim u_i(q_i)=0$.
But $u\equiv 1$ by
Lemma~\ref{maximum-principle}.
\end{proof}

\section{Circular Boundaries: Existence} \label{circular-section}

In this section, we 
fix an $R\ge 1$
and work in the ball $\overline{B_R}=\overline{B(0,R)}$. 
We let
\[
  D_R:= \{z=0\}\cap B_R.
\]

Note that the expander $\{\zeta=0\}$ 
(see Definition~\ref{zeta-definition}) is the horizontal plane
 $\{z=0\}$.
Note also, by the maximum principle, that any bounded region $\Omega$ of a level $\{\zeta=c\}$
is the {\bf unique} 
expander with boundary $\partial \Omega$.  In particular, $\Omega$ is the unique surface of least expander-area with boundary $\partial\Omega$.

We fix a $k\ge 1$.

\begin{definition} \label{def: C(s) and K(s)}
For $0<s<R$, let $C(s)$
be the union of the circles
of intersection of $\partial B_R$ with the expanders
$\{\zeta=s\}$,
 $\{\zeta=0\}$, 
 and
 $\{\zeta= (-s)\}$, and let
 $z(s)$ be the height of the circle $\{\zeta=s\}$.  
 Let $K(s)$ be the collection of all connected, embedded, $G_k$-invariant  minimal (for the expander metric) surfaces 
 $M\subset \overline{B_R}$ of genus $k-1$ such that
$\partial M=C(s)$ and such that $M$ is small and of type 1
in the sense
of Definitions~\ref{type-definition} and \ref{big-small-definition}.
\end{definition}

If $M\in K(s)$, we orient
$M$ so that if $\Omega$ is the connected component
 of $B_R\setminus M$
  containing $(0,0,-R)$ in its closure, then the unit normal to $M$ is the normal that points out of $\Omega$.  Equivalently, 
we give $M$ the orientation such that the induced orientation on $C(s)$ is 
clockwise on
the circle
 $\partial D_R$
and counterclockwise on
the circles
 $\{\zeta=\pm s\} \cap \partial B_R$.

Since the induced orientation on $\partial D_R$ is clockwise, 
\begin{equation}
\label{equator-normals}
\text{
$\nu(M,p)\cdot\ee_3<0$
for all $p\in \partial D_R$}.
\end{equation}

Our goal is the following:

\begin{theorem}
\label{fundamental-theorem}
There is an $\delta=\delta(k)>0$
such that for $0<s<\delta$,
 $K(s)$ contains exactly one surface $M$, and it is strictly $G_k$-stable for the expander metric.
\end{theorem}

We prove Theorem~\ref{fundamental-theorem}
by proving (for small $s$)
that 
 $K(s)$ contains a $G_k$-stable surface (Theorem~\ref{circle-existence-theorem}), that it contains only one surface (Theorem~\ref{uniqueness-theorem}),
and that the surface is strictly $G_k$-stable (Theorem~\ref{strict-theorem}.)

Recall that $Q_k\cap \partial D_r$ consists
of the $2k$ points in the 
circle $\partial D_r$
for which the polar coordinate $\theta$ is an odd
multiple of $\pi/(2k)$.
Thus the centers of the arcs $(\partial D_r)\setminus Q_k$
are the points in the  circle 
 $\partial D_r$ for which $\theta$
is a multiple of $\pi/k$.

\begin{definition}
\label{even-odd-definition}
Let $r>0$.
We say that a component
of $(\partial D_r)\setminus Q_k$ is an {\bf even arc}
or an {\bf odd arc}
according to whether, at the midpoint of the arc,
$\theta$ is an even or an odd multiple of $\pi/k$.
\end{definition}

\begin{lemma}[Topology Lemma]
\label{topology-lemma}
Suppose $M$ is a connected,
embedded, $G_k$-invariant expander in $\overline{B_R}$
with boundary $C(s)$
such that
  $M\cap\{z=0\}$
is
\begin{equation}
\label{zero-slice}
   (\partial D_R) \cup (D_R\cap Q_k)
\tag{*}
\end{equation}
and such that
\[
   \genus(M)\le k-1.
\]
Then 
\begin{enumerate}
\item 
\label{genus-item}
$M$ has genus $k-1$.
\item
\label{critical-point-item}
The only interior critical point of $\zeta(\cdot)|M$ is the origin,
and the only critical points
on $\partial M$ are the $2k$ points in $Q_k\cap\partial D_R$.
\item
\label{slice-item}
If $0< |t|<s$,   then
$M\cap\{\zeta=t\}$ is a simple 
closed curve.
\end{enumerate}
Furthermore, if 
\[
M^+:=M\cap\{z>0\} = M\cap\{\zeta>0\},
\]
then
\begin{equation*}
\partial M^+
=
C_s
\cup
(Q_k\cap D_R)
\cup 
\Psi,
\end{equation*}
where $C_s$ is the upper circle of $C(s)$
and where $\Psi$ consists of either all the even arcs or all the odd
arcs
 of $(\partial D_R)\setminus Q_k$.
(See Definition~\ref{even-odd-definition}.)
If $\partial M^+$ contains the even
arcs, then $M$ is of type 1,
and if $\partial M^+$ contains  the odd
arcs, then $M$ is of type 2.
\end{lemma}

\begin{proof}
Form a Riemann surface
$M'$ by doubling $M$ across the boundary component $\partial D_R$.
That is, we take a copy of $M$ and identify the points of $\partial D_R$ in $M$ with the corresponding points in the copy. Let $\iota:M'\to M'$ be the involution that is the identity on $\partial D$
and that maps each point
in $M\setminus \partial D_R$
to the corresponding point
in $M'\setminus M$.  
Extend the function   
 $\zeta(\cdot)$ on $M$
 to all of $M'$ by setting
 $\zeta(p)= - \zeta(\iota(p))$.

The origin $O$ and its image $O'$ under $\iota$  may be regarded as saddle  points of $\zeta(\cdot)$ of multiplicity $\ge k-1$.  
(See Remark~\ref{saddle-remark}.)
Each of the $2k$ points of $Q_k\cap\partial D_R$ is also
a saddle point of multiplicity $\ge 1$. 
Thus if $\sigma$ is the total
number
of saddles (counting multiplicity), then
\[
 \sigma 
 \ge 2(k-1) + 2k 
 = 4k - 2.
\]

Identify each of the four boundary components of $M'$ to a point to get a closed
surface $\overline{M'}$.

Now $z(\cdot)$ on $\overline{M'}$ has two maxima and two minima. 
Thus
\begin{align*}
(4k-2) - 4
&\le
\sigma - 4 
\\
&=
-\chi(\overline{M'})
\\
&=
2\genus(\overline{M'}) - 2
\\
&=
4\genus(M) - 2
\\
&\le
4(k-1)-2
\end{align*}
The first and last terms are equal, so the inequalities are
equalities.  In particular,
$M$ has genus $k-1$, and $O$, $O'$, and $Q_k\cup \partial D_R$
are the only saddle points on $M'$,
and therefore $O$ and 
 $Q_k\cup \partial D_R$
are the only critical points on $M$.
Consequently, 
$M^+$ 
and 
 $M^-:=M\cap\{z<0\}$
are annuli with no critical points of $\zeta|M$,
so $M\cap\{\zeta=t\}$
is a simple closed curve
if  $0<|t|<s$.

At this point, we have proved Assertions~\eqref{genus-item}, \eqref{critical-point-item},  and~\eqref{slice-item}.

Consequently, 
$M^+$ 
and 
 $M^-:=M\cap\{z<0\}$
are annuli with no critical points of $\zeta|M$,
so $M\cap\{\zeta=t\}$
is a simple closed curve
if  $0<|t|<s$.

Because the slice $M\cap\{\zeta=0\}$ is equal to~\eqref{zero-slice},
each component of 
 $(\partial D_R)\setminus Q_k$
 belongs to one and only of 
 $\partial M^+$ and $\partial M^-$.
By the $G_k$ symmetry,
$\partial M^+$ contains the even arcs and $\partial M^-$ contains the odd arcs, or vice versa.
Thus
\[
\partial M^+
=
C(s)^+ \cup (Q_k\cap D_R) \cup \Psi,
\]
where $\Psi$ is the union of the odd arcs or the union of the even arcs.
For the rest of the proof,
it suffices to consider the case
when $\partial M^+$ contains
the even arcs.

Let $\Gamma$ be 
the $\{y=0\}$ slice of $M$.
By the $(x,y,z)\mapsto (x,-y,z)$
symmetry, each component of $\Gamma$ is a smooth curve, and each critical point
of $\zeta(\cdot)|\Gamma$ is also
a critical point of $\zeta(\cdot)|M$.
Thus (by Assertion~\eqref{critical-point-item}) $O$ is the only critical point of $\zeta(\cdot)|\Gamma$.
Hence  $\zeta(\cdot)$ is strictly 
monotonic on each component of
$\Gamma\setminus\{O\}$.
In particular, $\Gamma$ contains no closed curves.

Let $\gamma$ be the interior of the component
of $\Gamma$ containing $(R,0,0)$.
Since we are assuming that $\partial M^+$ contains the even arcs, $\gamma$ lies
 in $\{z>0\}$.
Thus the other endpoint of $\gamma$ is either 
\begin{equation}
\label{the-other-endpoint}
   (\sqrt{R^2-z(s)^2},0,z(s))
\end{equation}
or
\begin{equation} \label{not-the-other-endpoint}
(-\sqrt{R^2-z(s)^2}, 0, z(s)),
\end{equation}
where $z(s)$ is the height
of the circle $\{\zeta=s\}$.
In fact, the other endpoint cannot be \eqref{not-the-other-endpoint}, because then
the component of 
\[
(\overline{B_R}\cap\{y=0\})
\setminus 
\overline{\gamma}
\]
that contains the
point~\eqref{the-other-endpoint} would contain
no other endpoint of $\Gamma$, which is impossible.
Thus $M$ is of type 1.
\end{proof}

\begin{remark}
\label{saddle-remark}
The notion of ``multiplicity of a saddle point'' is defined 
 in~\cite{hmw-morse}.
 In the language of that paper, 
the function $\zeta(\cdot)|M$ 
in the proof 
 of 
 Lemma~\ref{topology-lemma} 
 is a ``Morse-Rado function'', 
 by~\cite{hmw-morse}*{Theorem~8}.
Using that, it is straightforward to show that $\zeta(\cdot)$ is also a Morse-Rado function on $\overline{M'}$.
By~\cite{hmw-morse}*{Corollary~20},
the Euler characteristic is equal to the number of maxima and minima minus the number of saddles, counting multiplicity.
\end{remark}

\begin{proposition}
\label{smallish-lemma-alt}
Let $f$ be a smooth, nonnegative function on $\RR^3$ that is nonzero at some points of $D_R$.
Let $M_i$ be smoothly embedded expanders  in $\overline{B_R}$ with $\partial M_i=C(s_i)$, where $s_i\to 0$.
Suppose that
\[
\limsup_i\int_{M_i}f\,d\Hh^2
< 
2 \int_{D_R}f\,d\Hh^2.
\]
Then the $M_i$ converge
as varifolds to $D_R$ with multiplicity $1$,
and the convergence is smooth in compact subsets of 
  $B_R$.
In particular, if 
\[
  F: \RR^3\times \Ss^2\to \RR
\]
is a continuous function
such
 that $F(p,v)\equiv F(p,-v)$, then
\begin{equation}
\label{varifold-convergence}
\int_{M_i}F(p,\nu(p))\,d\Hh^2(p)
\to
\int_{D_R}F(p,\nu(p))\,d\Hh^2(p).
\end{equation}
\end{proposition}

\begin{proof}
In this proof, we do everything in terms of the Euclidean metric.
By the compactness theorem for integral varifolds, the $M_i$ converge weakly as varifolds to $D_R$ with some integer multiplicity $m$.
Now $O$ is in the support of $M_i$ for each $i$, so (by monotonicity) it is also in the support of $M$.  Thus $m\ge 1$.

By varifold convergence,
\begin{align*}
m \int_{D_R} f\,d\Hh^2
&=
\lim_i
\int_{M_i}f\,d\Hh^2 
\\
&<
2 \int_{D_R}f\,d\Hh^2.
\end{align*}
Therefore $m=1$.

By the Allard Regularity Theorem, the convergence is smooth
on compact subsets of $B_R$.  
Assertion~\eqref{varifold-convergence} holds by definition of varifold convergence.
\end{proof}

\begin{lemma}
\label{smallish-lemma}
Let $M_i$ be smoothly embedded expanders of uniformly bounded genus in $\overline{B_R}$ with $\partial M_i=C(s_i)$, where $s_i\to 0$.
Suppose that
\[
\limsup \area(M_i)<  2\area(D_R).
\]
or that 
\[
\limsup \int_{M_i}\phi\,d\Hh^2
< 2.
\] 
Then the $M_i$ converges smoothly to $D$ with multiplicity one on compact
subsets of $B_R$, and therefore
\[
\int_{M_i}\phi\,d\Hh^2 
\to 
1.
\]
\end{lemma}

In Lemma~\ref{smallish-lemma},
area refers to area with respect to the expander metric~$h$ (see~\eqref{expander-metric}), whereas $d\Hh^2$ is with respect to the Euclidean metric. The function $\phi$
is the function used to distinguish
big surfaces from small surfaces
 in~Definition~\ref{big-small-definition}.

\begin{proof}
Note that, for any surface $\Sigma$,
\[
\area(\Sigma)
=
\int_\Sigma f\,d\Hh^2,
\]
where
\[
  f(p)=e^{|p|^2/4}.
\]
Thus Lemma~\ref{smallish-lemma}
follows from~\eqref{varifold-convergence} by letting $F(p,v)=\phi(p)$.
\end{proof}

\begin{theorem}
\label{circle-existence-theorem}
If $s>0$ is sufficiently small,
then  $K(s)$ contains
a surface $M$ such that 
$M^+$ is stable.
\end{theorem}

\begin{proof}
Throughout this proof, ``area'' means ``area with respect to the expander metric $h$''. (See \eqref{expander-metric}).

Let $\Gamma$ be the union
of $Q_k\cap D_R$ and the even
arcs of $(\partial D_R)\setminus Q_k$. (See Definition~\ref{even-odd-definition}.)
For $0<\eps<1$, let
$\Gamma(\eps)$ be the simple
closed curve obtained from $\Gamma$ by removing $Q_k\cap B(0,\eps)$ and replacing
it by the odd arcs in  $\partial D_\eps\setminus Q_k$.
Note that $\Gamma(\eps)$ is a piecewise smooth, simple closed curve.

Let $C_s$ be the upper circle $\{z=s\}\cap \partial B_R$ of $C(s)$. 

Note that $D_R$ is the least-area surface bounded by $\partial D_R$.
For $0<s<R$, let $\Dd_s$ be the least-area surface bounded by $C_s$.
Note that
\[
   \lim_{s\to 0}
   \area(\Dd_s) = \area(D_R),
\]

The least area bounded by $\Gamma(\eps)$ is the area of the region $\Rr(\eps)$ in $\{z=0\}$ that it encloses, namely
\[
\area(\Rr(\eps))
=
\frac12\area(D_R) 
+ 
\frac12 \area(D_\eps).
\]
Now $C_s\cup \Gamma(\eps)$
bounds a piecewise smooth
annulus  consisting of 
\[
  \{0\le z\le s\} \cap\partial B_R
\]
together with
\[
   D_R\setminus \Rr(\eps).
\]
Let $\alpha(s,\eps)$ be the area
of that annulus.  Then
\begin{equation}
\label{annulus-area}
\begin{aligned}
\alpha(s,\eps)
&=
O(s) + \area(D_R)-\area(\Rr(\eps)
\\
&=
O(s) + \frac12\area(D_R)
-
\frac12\area(D_\eps).
\end{aligned}
\end{equation}

Thus, for small $s$,
\begin{align*}
\alpha(s,0)
&\cong 
\frac12\area(D_R),
\\
\area(\Dd_s)
+
\area(\Rr(0)) 
 &\cong
 \frac32\area(D_R).
\end{align*}
Hence, for all sufficiently small $s>0$, 
\[
\alpha(s,0)
<
\area(\Dd_s) + \area(\Rr(0)).
\]
Fix such an $s$.
For $\eps>0$, $\alpha(s,\eps)<\alpha(s,0)$
and $\area(\Rr(\eps))
>
\area(\Rr(0))$, so
it follows
that
\[
\alpha(s,\eps)
<
\area(D_s) + \area(\Rr(\eps)).
\]

Hence $C_s\cup \Gamma(\eps)$ bounds
an annulus $A(s,\eps)$ that minimizes
area among all annuli having
the symmetries of the boundary.

(In fact, an annulus that minimizes area among all annuli
will inherit the symmetries of the boundary, but we do not need that fact.)

Let $A'(s,\eps)$ be the image
of $A(s,\eps)$ under rotation by $\pi$ about one of the lines of $Q_k$.  The symmetries of $A(s,\eps)$ imply that the result
does not depend on the choice of which line.  Thus 
\[ 
   M(s,\eps):=A(s,\eps)
   \cup 
   A'(s,\eps)
\]
is a  smooth embedded minimal
surface whose boundary is $C(s)$
together with the horizontal circle
 $\partial D_\eps$.

One easily checks that the Euler characteristic of $M(s,\eps)$ is $-2k$, and thus
that the genus is $k-1$.

Standard compactness theorems 
(e.g., \cite{white18}) imply existence of $\eps_i\to 0$ such that
that the $M(s,\eps_i)$ converge
smoothly (and with multiplicity one) except at the origin to a limit surface $M=M(s)$.

Note that $\genus(M)\le k-1$.
Thus, by~\cite{gulliver},
$M$ is a smooth embedded minimal surface, including at the origin.

Since $\{\zeta=t\}\cap M(s,\eps_i)$
is nonempty for each $t\in [-s,s]$, it follows that $M$ has this property and hence that $M$ is connected.

Thus, $M$ has genus $k-1$,
 by Lemma~\ref{topology-lemma}.

By construction, the boundary of  $M^+$ is $C_s\cup \Gamma$.
Thus $M$ is of type 1 by 
  Lemma~\ref{topology-lemma}.

It remains only to show that $M$
is small (in the sense of Definition~\ref{big-small-definition}) when $s$ is small.

By construction and
 by~\eqref{annulus-area},
\[
\area(M(s)) 
\le 
2 \alpha(s,0)
\le
O(s) + \area(D_R),
\]
so
\[
\limsup_{s\to 0}\area(M(s))\le \area(D_R).
\]
Thus $M(s)$ is small for all sufficiently small $s>0$ 
by Lemma~\ref{smallish-lemma}.
\end{proof}

\section{Circular Boundaries:
Uniqueness and Strict Stability}
\label{uniqueness-section}
As in Section~\ref{circular-section},, we 
fix an $R\ge 1$
and work in the ball $\overline{B_R}=\overline{B(0,R)}$,
and we let $D_R$ denote the open disk $B_R\cap\{z=0\}$.  The curves $C(s)$ and surfaces $K(s)$ are specified in Definition~\ref{def: C(s) and K(s)}. In particular, the surfaces in $K(s)$ are of type 1. 

The goal of this section is to prove that when $s>0$ is sufficiently small, there is a unique surface in $K(s)$ and it is strictly $G_k$-stable.
The idea is fairly simple.
Suppose the strict stability were not true. (The uniqueness proof is similar.) Then there would be examples $M_i\in K(s_i)$ with $s_i>0$ and with nontivial $G_k$-invariant Jacobi-field functions $u_i$ that vanish at the boundary.  Let $p_i$ be a point where $|u_i(\cdot)|$ is a maximum.  One translates by $-p_i$ and then dilates by a well-chosen factor $\lambda_i$ to get surfaces $M_i'$ that converge subsequentially to a limit $M'$ having a Jacobi-field function $u'(\cdot)$ with certain properties. We prove the theorem by showing that $M'$ does not admit such a $u'(\cdot)$.  What makes the proof somewhat long is that we have to analyze the various surfaces $M'$ that can arise as blowups of the $M(s_i)$.

The next few lemmas 
and theorems describe
the behavior of $M_i\in K(s_i)$  and the behavior of blowups as $s_i\to 0$.
Note that any smooth limit of blowups will be minimal with respect to the Euclidean metric.

Roughly speaking, those lemmas and theorems say that if $M_i\in K(s_i)$ with $s_i\to 0$, then all the curvature concentrates at the points $Q_k\cap \partial D_R$.  If we blow up away from those points, we get flat limits, and if we blow up at one of those points, we get a specific curved limit minimal surface bounded by three parallel lines.

\begin{lemma}\label{disk-lemma}
Suppose $M_i\in K(s_i)$, where
  $s_i\to 0$.
Then $M_i$ converges smoothly to   $D_R$ with multiplicity $1$ away from $\partial D_R$.
\end{lemma}

This follows immediately from
Lemma~\ref{smallish-lemma}.

Recall (see Definition~\ref{def: C(s) and K(s)}) that for $0<s<R$, we let $z(s)$ be the height of the circle 
\[
\{\zeta=s\}\cap \partial B_R.
\]
Thus the circles
$
\{z=z(s)\}\cap
  \partial B_R
$
and
$
\{\zeta= s\}\cap
  \partial B_R.
$
are the same.
Note that $s_i\to 0$ if and only if $z(s_i)\to 0$.

\begin{lemma}
\label{cylinder-lemma}
Let $Z(r)$ be the closed solid cylinder of radius $r$ with axis $Z$, and let $r(M,\eps)$ be the supremum
of $r>0$ such that 

\[
   M\cap Z(r)
\]
 is the graph of a function
\[
   f: \overline{D_r}\to \RR
\]
with $\sup|Df|\le \eps$.
For every $\eps>0$ there is a $c=c_\eps<\infty$ 
with the following property.
If $s\in (0,1)$
and $M\in K(s)$, then
\[
   r(M,\eps)> R - cz(s).
\]
\end{lemma}

\begin{proof}
Suppose not. Then there exist examples $M_i\in  K(s_i)$ with
\[
   \frac{R-r(M_i,\eps)}{z_i} \to\infty,
\]
where $z_i:=z(s_i)$.
Thus $z_i\to 0$.

By Lemma~\ref{disk-lemma}, 
\begin{equation}
\label{radius-to-one}
   \lim_i r(M_i,\eps)=R.
\end{equation}

Note that there is a point $p_i\in M_i$ such that
\[
  \dist(p_i,Z)=r(M_i,\eps)
\]
and such that the slope of the tangent plane to $M_i$ at $p_i$ is $\ge \eps$.
Translate $M_i$ 
and $H_i:=Z(r(M_i,\eps))$ 
by $-p_i$
and dilate by $z_i^{-1}$
to get $M_i'$ and $H_i'$.
Note that
\[
  \dist(0,\partial M_i')\to\infty.
\]
Also, $M_i'$ is contained in a horizontal slab of thickness $2$.
By~\eqref{radius-to-one}, the radius of the solid cylinder 
 $H_i'$
tends to $\infty$ as $i\to\infty$.

After passing to a subsequence,
the $H_i'$ converge 
to a halfspace bounded by a vertical plane through the origin, and $M_i'$ converge to a
complete minimal surface $M'$ of finite total curvature contained in a horizontal slab.  Thus $M'$ is a union of horizontal planes, possibly with multiplicity.
(If the multiplicity is $>1$, the convergence might fail to be smooth at a finite set of points.)
But 
\[
  M'\cap (H'\setminus \partial H')
\]
is a multiplicity-one graph.
Thus $M'$ is a single,  multiplicity-one, horizontal plane, and the convergence $M_i'$ to $M'$ is smooth everywhere.
But that is impossible since the slope of $\Tan(M_i',0)$ is $\ge \eps$
for each $i$.
\end{proof}

\newcommand{\wtB}{\widetilde{B}}

Lemmas~\ref{disk-lemma}
and~\ref{cylinder-lemma}
describe the behavior away from the boundary of $M\in K(s)$ for $s$ small.
 We now turn our attention to the behavior near the boundary.
By the $G_k$ symmetry, it suffices to consider points
in the wedge region given in cylindrical coordinates by
\[
  0\le \theta \le \frac{\pi}{2k}.
\]
Recall that $Q_k$ consists
of the points for which $z=0$ and which $\theta$ is an odd multiple of $\pi/(2k)$. 
 The wedge region is bounded
 by $\{y=0\}$ (which is a plane of reflectional symmetry), and the vertical plane $\theta=\frac{\pi}{2k}$, which contains a line in $Q_k$. Therefore the intersection of this wedge region with $\partial D_R\cap Q_k$ is a single point, the point with cylindrical
  coordinates~$(R,\frac{\pi}{2k},0)$. 

\begin{lemma}
\label{barrier-lemma}
Suppose $M$ is a minimal surface that is contained in a half-slab and whose   
boundary is contained in a plane $P$ parallel to the half-plane boundaries of the half-slab. Then $M$ is contained in $P$.
\end{lemma}

\begin{proof}
By translating and rotating, we can assume that the plane is the plane $\{z=0\}$ and that the half-slab is disjoint from a catenoid $C$ centered at the origin and with axis $Z$.  By the strong maximum principle, the set of $t\in (0,1]$ such that $tC$ is disjoint from $M$ is an open and closed subset of $(0,1]$ that contains $1$, and hence contains all $t\in (0,1]$.
\end{proof}

\begin{theorem}
\label{halfplane-strip-theorem}
Let $M_i\in K(s_i)$ with $s_i\to 0$.
Suppose that $q_i\in \partial D_R$ converges
to a point $q$,
that
\[
  0
  \le 
  \theta(q_i)
  \le 
  \frac{\pi}{2k},
\]
and that
\[
  \lim \frac{\dist(q_i,Q)}{z_i}
=
  \infty,
\]
where $z_i=z(s_i)$.
Then
\[
\wtB_i:= (1/z_i)(B_R-q_i)
\]
converges to the halfspace 
\[
 H= \{ v: q\cdot v\le 0\},
\]
and
\[
   \wtM_i:=  (1/z_i)( M_i - q_i)
\]
converges smoothly and with multiplicity one to  
the union of the halfplane
\[
    H\cap \{z=(-1)\}
\]
and the  strip
\[
   P \cap \{0\le z\le 1\},
\]
where $P$ is the plane
\[
 P= \partial H = \{v: 
 q \cdot v=0\}.
\]
\end{theorem}

\begin{proof}
Convergence
 of $\wtB_i$
to the halfspace $H$ is
 trivial.
For $t\in \RR$, let
$L_t$ be the line
\[
L_t
:=
P\cap \{z=t\}.
\]

Let 
$\wtM_i^+$
be the closure of
$\wtM_i\cap\{z>0\}$
and $\wtM_i^-$ be the closure
of $\wtM_i\cap\{z<0\}$.

By Lemma~\ref{topology-lemma}, 
$\partial \wtM_i^+$
converges to $L_0\cup L_1$
and $\partial\wtM_i^-$
converges to $L_{-1}$.
The convergence is smooth and with multiplicity one.

It follows that $\wtM_i^-$
and $\wtM_i^+$ converge
(perhaps after passing to a subsequence)
 smoothly and with
multiplicity one to  limits 
$\wtM^-$ and $\wtM^+$.
Now $\wtM^-$ has boundary $L_{-1}$ and  is contained
in the horizontal half-slab
$H\cap \{-1\le z\le 0\}$.
By Lemma~\ref{barrier-lemma},
 $\wtM^-$ is a horizontal half-plane.

By Lemma~\ref{cylinder-lemma}, there is a $c<\infty$ such that each vertical line $L$ in $H$
with $\dist(L,P)>c$ intersects $\wtM:=\wtM^+\cup \wtM^-$ 
 at most once.
Since it intersects $\wtM^-$,
it must not intersect $\wtM^+$.

Hence $\wtM^+$ lies within a bounded distance of the strip
\[
   P \cap \{0\le z\le 1\}.
\]
Catenoidal barriers force  $\wtM^-$ to be that strip.

We have shown subsequential converge of $\wtM_i$ to the indicated limit.  Since the limit is independent of the choice of subsequence, the original sequence converges to that limit.
\end{proof}

\begin{corollary}
\label{halfplane-strip-corollary}
Suppose that $M_i\in K(s_i)$, where $s_i\to 0$, that $p_i\in M_i$,
that
\[
  0\le \theta(p_i)\le \pi/k,
\]
and that 
\[
   \frac{\dist(p_i,Q_k\cap \partial D_R)}{z_i} \to \infty.
\]
where $z_i=z(s_i)$.
Let  $R_i=\sqrt{R^2-z_i^2}$
be the radius of the circle
$\{\zeta=s_i\}\cap\partial B_R$.
Then, after passing to a subsequence, either
\begin{equation}
\text{$\nu(M_i,p_i)\to \ee_3$ 
and 
$p_i\in
Z(R_i)$},
\end{equation}
or
\begin{equation}
\begin{gathered}
\text{$p_i\to p$ 
for some $p\in \partial D_R$, 
$\nu(M_i,p_i)\to -p/R$, 
and}
\\
\limsup_i
\frac{\dist(p_i,\partial D_R)}{z_i} \le  1.  
\end{gathered}
\end{equation}
\end{corollary}

\begin{proof}
We may assume that 
\[
  \frac{(R-|p_i|)}{z_i} \to \sigma \in [0,\infty].
\]
If $\sigma=\infty$, then $\nu(M_i,p_i)\to\ee_3$ by Lemma~\ref{cylinder-lemma}.
If $\sigma<\infty$, then the conclusion follows immediately from
 Theorem~\ref{halfplane-strip-theorem}.  
\end{proof}

\begin{corollary} 
\label{gauss-bonnet-corollary} Suppose $M_i\in K(s_i)$, and $s_i\to 0$. Then
\[
\lim_i \frac12\int_{M_i}|A_i|^2\,d\Hh^2 = 
4\pi k,
\]
 where $A_i$ is the (Euclidean) second fundamental form of $A_i$, and thus $|A_i|^2$ is the sum of the squares of the principal curvatures.
\end{corollary}

\begin{proof}
The Euler characteristic of $M_i$ is $1-2k$, so,
by Lemma~\ref{gauss-bonnet-lemma},
\begin{equation}
\label{george}
\begin{aligned}
\frac12
\int_{M_i}
|A_i|^2\,d\Hh^2
&=
\frac18\int_{M_i}(\nu_i\cdot p)^2\,d\Hh^2p
+
\int_{C(s_i)}
\kk_i\cdot\nn_i\,d\Hh^1
\\
&\quad
+ 2\pi(2k-1).
\end{aligned}
\end{equation}
By~\eqref{varifold-convergence},
\begin{equation}
\label{rondi}
\lim_i\int_{M_i}(p\cdot\nu_i)^2\,d\Hh^2p
=
\int_{D_R}(p\cdot \nu)^2\,d\Hh^2p
= 0.
\end{equation}

Let 
\[
0<\alpha<\frac{\pi}{2k}.
\]
By Theorem~\ref{halfplane-strip-theorem}, $\kk_i\cdot \nn_i$
converges uniformly to $0$ on
the $0<\theta<\alpha$ portions of the $\{\zeta=s_i\}$
and $\{\zeta=0\}$ circles of $C(s_i)$, and it converges
to $1$ on that portion of the 
 $\{\zeta=-s_i\}$ circle.

Thus
\[
\lim_i 
\int_{\partial M_i,\, 0\le\theta\le\alpha}
\kk_i\cdot\nn_i\,d\Hh^1
=
\alpha.
\]
Thus, letting $\alpha\to \pi/(2k)$,
\[
\lim_i 
\int_{\partial M_i,\, 0\le\theta\le\pi/(2k)}
\kk\cdot\nn\,d\Hh^1
=
\frac{\pi}{2k}.
\]
By the $G_k$-symmetry,
\[
\lim_i
\int_{\partial M_i}
\kk\cdot\nn\,d\Hh^1
=
2\pi.
\]
Combining this with~\eqref{george} 
and~\eqref{rondi} gives Corollary~\ref{gauss-bonnet-corollary}
\end{proof}

Let $\Ll$ be the union of the vertical lines through the $2k$ points of 
 $Q\cap \partial D_R$.
Let $\Ll(r)$ be the union of the closed solid cylinders
of radius $r$ about each of the lines in $\Ll$.

\begin{theorem}
\label{structure-theorem}
There is a $c<\infty$ with the following property.
If $s>0$ is sufficiently small and if $M\in K(s)$, then 
one component, $\wtM$,
of $M\setminus \Ll(c z(s))$
is the graph of a function
\[
  f: \overline{D_{R(s)}}\setminus \Ll(c z(s))
  \to \RR,
\]
where $R(s)=\sqrt{R^2-z(s)^2}$ is the radius
of the circle 
\[
  C_s 
  := 
  (\partial B_R)
  \cap
  \{\zeta=s\}.
\]
Furthermore,
\begin{enumerate}
\item 
$M\setminus \wtM$ lies
within distance $cz(s)$ of $\partial D_R$.
\item
$|\nu(M,\cdot)-\ee_3|<\frac15$
at all points of $\wtM$,
\item At all points $p$ of 
$M\setminus (\wtM\cup \Ll(cz(s)))$, 
\[
  |\nu(M,p)\cdot\ee_3|< \frac15.
\]
\end{enumerate}
\end{theorem}

\begin{proof}
By Corollary \ref{halfplane-strip-corollary}, there exist $c<\infty$ and $\delta>0$ such 
that for $M\in K(s)$ with $s<\delta$, 
if $p\in M\setminus \Ll(cs)$, then 
\begin{equation}
\label{graph-part}
\text{$|\nu(M,p)-\ee_3|< \frac15$
and $p\in Z(R(s))$}
\end{equation}
or
\begin{equation}
\label{ribbon-part}
\begin{gathered}
\text{$|\nu(M,p)\cdot\ee_3|<\frac15$, and} \\
\dist(p,\partial D_R)\le 2 z(s) < R/2.
\end{gathered}
\end{equation}

It follows that on each component of $M\setminus\Ll(cz(s))$, either all points satisfy~\eqref{graph-part} or
all points
satisfy~\eqref{ribbon-part}.

By replacing $\delta$ by a smaller $\delta>0$, we can assume that
$\Ll(cz(s))$ consists of $2k$ disjoint cylinders, and thus that
\[
\Omega:=
\overline{D_{R(s)}}
\setminus 
\Ll(cz(s))
\]
is connected.  It follows that each component of $\MM\setminus\Ll(cz(s))$ that satisfies~\eqref{graph-part}
is the graph of a function
over $\Omega$.

By Lemma~\ref{disk-lemma}, we can also assume that $Z$ intersects $M$ in exactly one point, and that
\[
  |\nu - \ee_3| < \frac15
\]
at that point. It follows that there is exactly component of $\MM\setminus\Ll(cz(s))$ whose points satisfy~\eqref{graph-part}.
Let $\wtM$ be that component.
\end{proof}

\begin{theorem}
\label{nonflat-blowup-theorem} Let $M_i\in K(s_i)$ with $s_i\to 0$, and let $z_i=z(s_i)$.
Let $q$ be the point
on the circle $\partial D_R$ with
\[
  \theta(q)=\frac{\pi}{2k}.
\]
(Thus
 $q\in Q_k\cap \partial D_R$.)
Rotate $M_i$ by $-\pi/(2k)$ about $Z$, then 
translate by $-R\ee_1$, and then dilate by $1/z_i$ to get $\wtM_i$.

Then 
$
\wtM_i
$
converges smoothly 
(with multiplicity one) to a connected, embedded, oriented  minimal surface 
$M$ with the following properties:
\begin{enumerate}
\item
\label{halfslab}
$M$ is contained
 in the half-slab
 $\{x\le 0\}
 \cap
 \{-1\le z \le 1\}
 $, and $M$ contains 
  $X^-:=\{(x,0,0):x\le 0\}$.
 \item 
\label{sliced}
The boundary of
$M$ is the union 
of the three lines $Y$,
$Y+\ee_3$, and $Y-\ee_3$.
\item
\label{2pi}
$\int_M |K|\,d\Hh^2\le \pi$.
\item 
 $M\setminus X^-$ is not connected.
\item $M$ contains a point
$p$ at which $\nu\cdot \ee_2<0$,
where $\nu$ is the unit normal vectorfield on $M$
such that $\nu\cdot\ee_3>0$ on $Y+\ee_3$.
\end{enumerate}
\end{theorem}

\begin{proof}
Note that the boundaries
 $\partial \wtM_i$
 converge smoothly (with multiplicity one)
to the union of $Y$, $Y+\ee_3$, and $Y-\ee_3$.

Now choose $r>0$ small enough that
the balls of radius $r$
about the $2k$ points in $Q_k\cap \partial D$ are disjoint.  
By symmetry,
\begin{align*}
(2k)\int_{B(q,r)\cap M_i}|K|\,d\Hh^2
&\le
\int_{M_i}|K|\,d\Hh^2
\\
&=
\int_{M_i}|K|\,d\Hh^2
\\
&\to 
4\pi k
\end{align*}
 by Corollary~\ref{gauss-bonnet-corollary}.
Thus
\[
 \limsup_i 
 \int_{B(q,r)\cap M_i}|K|\,d\Hh^2
 \le
 2\pi,
\]
so, after passing to a subsequence, $\wtM_i$ converges
 to a smooth limit surface $M$, possibly with multiplicity.  
The convergence to any components of $M$ with multiplicity is smooth,
and the convergence to higher multiplicity components (if there are any)
is smooth away from a finite 
set of points.

Note that $M$ is contained
in the half-slab indicated in Assertion~$(1)$, and $M$ contains $X^-$.

By Lemma~\ref{cylinder-lemma}, there is a $c<\infty$ such that one component
of $\wtM\setminus Z(c)$
is the (multiplicity one) graph of a function
\[
  f: \{y\le 0\}\setminus 
   \overline{D(c)} \to \RR
\]
such that
$f(0,y)=-1$ for $y<-c$
and $f(0,y)=1$ for $y>c$.
By symmetry, $f(x,0)=0$
for $x<-c$.

Let $M'$ be the component
of $M$ that contains the graph of $f$.
Now $M'$ contains portions of $Y-\ee_3$ and $Y+\ee_3$,
and therefore $M'$ contains all of those two lines.
Note that $M'$ is symmetric 
under rotation about $X$,
so it contains all of $X^-$.
In particular, it contains $O$ and therefore it contains all of $Y$.
Thus it is all of $M'$, since if there were any other component, that component would be a complete minimal surface (without boundary)
of finite total curvature
contained in the half-slab, which is impossible.

Thus $M=M'$ is connected and has multiplicity one, so the convergence to $M$ is smooth
everywhere.

Note that $\nu\cdot\ee_3>0$
on the graph of $f$.
Let $R>c$.
As $\theta$ goes from $\pi/2$
to $3\pi/2$, 
\[
   f(R\cos\theta, R\sin\theta)
\]
goes from $1$ to $-1$.
Hence there is a point 
$(R\cos\theta, R\sin\theta)$
with 
\[
 \tfrac{\pi}2 
 <
 \theta
 <\tfrac{3\pi}2
\]
at which 
\[
 \frac{\partial f}{\partial \theta}< 0.
\]
At the corresponding point
on $M$,  $\nu\cdot\ee_2<0$.

So far, we have shown only subsequential convergence
of  $\wtM_i^+$ to a limit surface $M$ having the asserted properties.
But Theorem~\ref{unique-limit-theorem} below implies that the limit 
is independent of the choice
of subsequence, and hence the original sequence converges to $M$.
\end{proof}

\begin{corollary}
\label{strictly-stable-corollary}
Each component of $M\setminus X^-$ is strictly stable.
\end{corollary}

\begin{proof}
By \cite{barbosa-docarmo}, any orientable minimal surface in $\RR^3$ of total curvature $< 2\pi$ is strictly stable.
By symmetry, each component
of $M\setminus X^-$ has
total curvature $\le \pi$.
Hence it is strictly stable.
\end{proof}

\begin{theorem}
\label{unique-limit-theorem}
There is a unique embedded minimal surface $M$ in the halfslab
\[
   \{x\le 0\}
   \cap
   \{-1\le z \le 1\}
\]
such that 
\begin{enumerate}
\item 
$\partial M
=
Y
\cup 
(Y+\ee_3)
\cup 
(Y-\ee_3)$.
\item 
$\int_M|K|\,d\Hh^2\le 2\pi$,
\item $\nu\cdot \ee_2<0$ 
at some point of $M$, where
$\nu$ is the unit normal vectorfield on $M$ such that 
$\nu\cdot\ee_3>0$ on $Y+\ee_3$.
\item $M$ contains $X^-$.
\end{enumerate}
\end{theorem}

\begin{proof}
Existence follows from Theorem \ref{nonflat-blowup-theorem},
so it suffices to prove 
uniqueness.

\setcounter{claim}{0}
\begin{claim} \label{claim-nu}
If $p_i\in M$ diverges, then $\nu(p_i)$ converges ({subsequentially})
to $\pm \ee_1$ or $\pm \ee_3$.
\end{claim}

\begin{proof}
$M_i-p_i$ converges smoothly, after passing to a subsequence, to a flat surface in a horizontal slab
with no boundary or with boundary consisting of three parallel horizontal lines in a vertical plane $\{y=c\}$.  Each component of such a surface is either a plane, a half-plane, or a strip in a vertical plane.
\end{proof}

Let $C$, $H^+$, and $H^-$ be
the circle and the hemispheres in $\Ss^2$ given by
\begin{align*}
C&=\Ss^2\cap \{y=0\}, \\
H^+&= \Ss^2\cap \{y>0\}, \\
H^- &= \Ss^2\cap\{y<0\}.
\end{align*}

Note that 
\[
\nu(\partial M)\subset C.
\]
Thus, by Claim \ref{claim-nu},
$\nu$ is a proper map
from $\nu^{-1}(H^\pm)$
to $H^\pm$, so it has a mapping
degree $d_\pm$. 
Now
\begin{align*}
2\pi
&\ge
\int_M|K|\,d\Hh^2
\\
&=
d_+ \area(H^+) 
+ d_-\area(H^-)
\\
&=
2\pi (d_+ + d_-).
\end{align*}
By hypothesis {\it (3)}, $d_-$ is nonzero,
so $d_+=0$ and $d_-=1$.

Thus $\nu$ maps $M\setminus\partial M$ diffeomorphically onto $H^-$.

It follows that $\nu$ maps 
\[
  Y \cup (Y+\ee_3) \cup
  (Y - \ee_3)
\]
diffeomorphically onto the union of three disjoint arcs of $C$.

Since $\nu\cdot \ee_3<0$
at the origin, we see from the strong maximum principle
that $\nu\cdot\ee_3<0$ at all points of $Y$. Thus $\nu$ maps $Y$ diffeomorphically onto an arc
of 
\[
  J_0:= C\cap \{z<0\}
\]
By Claim \ref{claim-nu} (and the fact that
$-\ee_3=\nu(O)\subset \nu(Y)$), the endpoints of $\nu(Y)$ can only be $\pm \ee_1$. Thus
\[
  \nu(Y)=J_1.
\]
Since $\nu\cdot \ee_3<0$
in $Y$, it follows that 
$\nu\cdot\ee_3>0$ at all points of $Y\pm \ee_3$.

By the strong maximum principle $\nu\cdot \ee_1>0$
at all points of $Y-\ee_3$.
Thus $\nu(Y-\ee_3)$ is contained
in 
\[
  J_{-1}:= C\cap\{x>0,\, z>0\}.
\]
By Claim \ref{claim-nu}, the endpoints of $\nu(Y-\ee_3)$ must belong
to the set $\{\pm\ee_1, \pm \ee_3\}$.
 Hence
\[
   \nu(Y-\ee_3)= J_{-1}.
\]
By the same reasoning,
\[
  \nu(Y+\ee_3)
  =
  J_1:= C\cap\{x<0, \, z>0\}.
\]
Hence $$z\circ \nu^{-1}:H^- \longrightarrow \RR$$
is the unique bounded harmonic function on $H^-$ with boundary values 
$1$ on $J_1$, $0$ on $J_0$,
and $-1$ on $J_{-1}$.

Hence, by the theory of the Weierstrass representation, 
$M$ is determined.
\end{proof}

\newcommand{\tM}{\tilde M}

\begin{theorem}
\label{uniqueness-theorem}
For all sufficiently small $s>0$, there is at most one surface in $K(s)$.
\end{theorem}

\begin{proof}
Suppose not.
Then there exist $s_i\to 0$
such that $K(s_i)$ contains two 
distinct surfaces $M_i$ and $N_i$.

Let $p_i$
be a point in $M_i$ at which
the function
\[
   p\in M_i\mapsto \dist(p,N_i)
\]
attains its maximum.
By symmetry, we can assume that
\[
  0\le \theta(p_i)\le \tfrac{\pi}{2k}.
\]

We may assume that $p_i$ converges to a limit $p\in \overline{D_R}$.

Note that $M_i$ and $N_i$
converge smoothly and with multiplicity one in $B_R$ to the disk   $D_R$.  Also, $0\in M_i\cap N_i$.
Thus,
 by Corollary~\ref{pair-corollary}, 
\[
  p\in \partial D_R.
\]

Now $s_i\to 0$, so $z_i:=z(s_i)\to 0$, and, thus, we can assume that
\[
   cz_i \ll R,
\]
where $c$ is as in Lemma~\ref{cylinder-lemma}.
We let $\wtM_i$ be the component
of $M_i\setminus \Ll(cz_i)$
that projects diffeomorphically
onto
 $\overline{D_{R(s_i)}}\setminus \Ll(cz_i)$,
and let $\wtN_i$ be the corresponding component of $N_i$.

Next, we claim that
\begin{equation} \label{eq:uno}
  \limsup \frac{R-|p_i|}{z_i}<\infty.
\end{equation}
For suppose to the contrary that
\begin{equation} \label{eq:nouno}
 \limsup \frac{R-|p_i|}{z_i}=\infty.
\end{equation}

Translate $D_R$, $\Ll(cz_i)$, 
 $\wtM_i$ and $\wtN_i$
by $-p_i$ and then dilate
by $1/(R-|p_i|)$
to get $D_i'$, $\Ll_i'$,
$\wtM_i'$, and $\wtN_i'$.
After passing to a subsequence,
$D_i'$ converges to a horizontal
halfplane $D'$ and $\Ll_i'$ converges with to the empty set or to a single vertical line.

By \eqref{eq:nouno}, $\wtM_i'$ and $\wtN_i'$ converge to $D'$, and the convergence
is smooth with multiplicity one everywhere if $\Ll'=\varnothing$
and everywhere except at the point $D'\cap \Ll'$ if $\Ll'$ is nonempty.

But that contradicts Corollary~\ref{pair-corollary}.
Thus \eqref{eq:uno} holds.

Let $q$ be the point on $\partial D$ for which
\[
 \theta(q)=\tfrac{\pi}{2k}.
\]
By passing to a subsequence, we can assume that
\[
  \frac{\dist(p_i,q)}{z_i}
  \to \mu \in [0,\infty].
\]

We claim that $\mu<\infty$.
For suppose, to the contrary, that $\mu=\infty$.
Let $q_i$ be the point in $\partial D$ closest to $p_i$.

Now translate $p_i$, $B_i$, $M_i$, and $N_i$ by $-q_i$ then dilate by 
 $1/z_i$ to get $p_i'$, $B_i'$,
 $M_i'$, and $N_i'$  

By Theorem~\ref{halfplane-strip-theorem}, $B_i'$ converges to a halfspace $H'$ bounded by a vertical plane $P$, and $M_i'$ and $N_i'$ both converge smoothly and with multiplicty one
to the union of the halfplane
\[
  \{z=-1\}\cap H'
\]
and the strip
\[
   \{0\le z \le 1\}\cap P.
\]
By~\eqref{eq:uno}, the $p_i'$ converge subsequentially to a limit $p'$.
But that contradicts 
Corollary~\ref{pair-corollary}.

Thus
\[
  \frac{\dist(p_i,q)}{z_i}
  \to 
  \mu
  < 
  \infty.
\]
Now translate $p_i$, $M_i$, and $N_i$ by $-q$ and then dilate
by $1/z_i$ to get 
 $p_i'$, $M_i'$, and $N_i'$.

Now $|p_i'|\to \mu<\infty$, so (after passing to a subsequence)
$p_i'$ converges to a limit $p'$.
By Theorems~\ref{nonflat-blowup-theorem} and~\ref{unique-limit-theorem}, $M_i'$ and $N_i'$
converge smoothly (with multiplicity $1$) to the same 
surface $\Sigma$.

By Lemma~\ref{pair-lemma}, there is a Jacobi-field function $u$ on $\Sigma$ such that $u\equiv 0$ on $\partial\Sigma$ and
such that
\[
  \max |u(\cdot)|
  =
  u(p') = 1.
\]
 Note that $u\equiv 0$ on $X^-$.
But, by Corollary~\ref{strictly-stable-corollary}, $\Sigma\setminus X^-$ is strictly stable, and thus there is no such $u$.
\end{proof}

\begin{theorem}
\label{strict-theorem}
For all sufficiently small $s>0$,
$K(s)$ contains exactly one surface $M$, and it is strictly $G_k$-stable.
\end{theorem}

\begin{proof}
By Theorem~\ref{uniqueness-theorem}, we know that for all sufficiently small $s$, $K(s)$ has exactly one surface $M$, and it is stable.
Thus it suffices to show  (if $s$ is small)
that $M$ has no nontrivial $G_k$-invariant normal Jacobi field that vanishes at the boundary.

Suppose not.
Then there exist $s_i>0$
converging to $0$, $M_i\in K(s_i)$, and nonzero $G_k$-invariant Jacobi fields
\[
   p\mapsto u_i(p)\nu(M_i,p)
\]
that vanish at the boundary.
Let $p_i$ be a point
such that
\[
  \max |u_i(\cdot)| = |u(p_i)|.
\]

Now we get a contradiction exactly as in the proof of Theorem~\ref{uniqueness-theorem}, 
invoking 
Lemma~\ref{jacobi-lemma}
and Corollary~\ref{jacobi-corollary} in place
of 
Lemma~\ref{pair-lemma}
and Corollary~\ref{pair-corollary}.
\end{proof}

...

\section{Big Surfaces} \label{big-section}

\begin{theorem} 
\label{big-theorem}
Under the hypotheses of Theorem~\ref{main-theorem},
\[
   d_{\textnormal{big},i} +
   d_{\textnormal{small},i}
=
 0
\]
for $i=1,2$.
\end{theorem}

\begin{proof}
We prove the theorem for $i=1$.
The proof for $i=2$ is identical.

We will apply Theorem~\ref{basic-theorem}, but with a different choice of $S$ and $K$ than in Section~\ref{applying-section}.  

As in Section~\ref{circular-section}, 
For $0<s<R$, we let $C(s)$
be the union of the circles of intersection of $\partial B(0,1)$
with the  $\{\zeta=s\}$,
$\{\zeta=0\}$, and $\{\zeta=(-s)\}$.

We let $S'$ be the set of $C(s)$, $0<s<R$.
We let $K'$ be the set of all $M$ in $\MM(S')$
such that $M$ is of type 1.

Then, as in Section~\ref{applying-section},
one easily checks that $(S',K')\in \Ff$, so it has a degree $d'=\deg(S',K')$.

\setcounter{claim}{0}
\begin{claim}
If $s$ is sufficiently close
to $R$, then $C(s)$ bounds no surface in $K'$.
\end{claim}

\begin{proof}
Suppose not. Then there exist $s_i\to R$ and surfaces $M_i$ in $K'$ with $\partial M_i=C(s_i)$.
After passing to a subsequence, the $M_i$ converge to an $h$-minimal surface $M'$ with boundary contained in the union
of the equator $\partial D_R$ and the two points $(0,0,\pm R)$.
Since $M_i\cap \{\zeta=t\}$
is an annulus for $0<t<s_i$,
we see that $M\cap\{\zeta=t\}$ is an annulus
for $0<t<R$.
Thus $M\cap\{z>0\}$ is a punctured disk.
By~\cite{gulliver}, $(0,0,R)$ is a removable singularity.
But then the strong maximum is violated, since $\zeta(\cdot)|M$ attains its maximum at $(0,0,R)$.
\end{proof}

Thus
\begin{equation}\label{dprimezero}
   d' = 0.
\end{equation}

Now fix an $\tilde s>0$ small enough that $C(\tilde s)\in S$,
where $S$ is as 
 in Definition~\ref{Sc-definition}.
Write $C=C(\tilde s)$.

\newcommand{\bg}{_\textnormal{big}}
\newcommand{\smll}{_\textnormal{small}}

By~\eqref{dprimezero} and 
 Theorem~\ref{basic-theorem}\eqref{pointwise-item}, 
\begin{equation}\label{morezero}
0 = d' = \deg(C , K'(C)).
\end{equation}
Let $K'\bg(C)$
be set of elements of $K'(C)$ that are big,
and $K'\smll(C)$
be the set of elements of $K'(C)$ that are small.
Thus $K'(C)$ is the disjoint union
of 
$K'\bg(C)$
and $K'\smll(C)$.
Because $s$ is small, 
 $K'\bg(C)$
and $K'\smll(C)$
are relatively open-and-closed    subsets of $K'(C)$,
by Theorem~\ref{epsilon-theorem}.

Thus, by~\eqref{morezero}
and Theorem~\ref{basic-theorem}\eqref{additivity-item}, 
\begin{equation}
\label{big-and-small}
\begin{aligned}
0
&=
\deg(C, K'(C))
\\
&=
\deg(C, K'\bg(C))
+
\deg(C, K'\smll(C)).
\end{aligned}
\end{equation}

Note that the sets
 $K'\bg(C)$ 
 and
 $K'\smll(C)$
are identical
to the sets
$K\bone(C)$
and
$K\sone(C)$, respectively,  
in Definition~\ref{def:K's}.

Thus we can rewrite~\eqref{big-and-small} as
\begin{equation}
\label{almost-there}
0 
=\deg(C,K\bone(C))
+ \deg(C, K\sone(C)).
\end{equation}

Finally,
\begin{equation}
\label{finally}
\begin{aligned}
d\bone &= \deg(C, K\bone(C)), 
\\
d\sone &= \deg(C, K\sone(C)),
\end{aligned}
\end{equation}
by Theorem~\ref{basic-theorem}\,\eqref{additivity-item}.

Thus 
\[
 d\bone + d\sone = 0
\]
by~\eqref{almost-there}
 and~\eqref{finally}.
\end{proof}

\begin{corollary} 
\label{big-corollary}
Suppose $C\in \Sc$,
  where $\Sc$ is as in Definition~\ref{Sc-definition}.
Then $C$ bounds at least
one big, $G_k$-invariant
expander of type 1
and at least one big, $G_k$-invariant expander
of type 2.
\end{corollary}

\begin{proof}
  By Theorem~\ref{small-theorem} and equation \eqref{eq:symmetry}, $\dsone=\dstwo\neq 0$.
Hence, by Theorem~\ref{big-theorem}, 
$\dbone=\dbtwo\ne 0$.
The corollary then follows
immediately
from Theorem~\ref{basic-theorem}\,\eqref{empty-item}.
\end{proof}

\section{Handles Cannot Escape} \label{noescape-section}

In this section, we show that when we let $R\to \infty$, 
we get subsequential convergence to a smooth expander of the same
genus and asymptotic to the given cone $C$.
The arguments in this section hold in general dimensions and codimensions.

\begin{lemma}
\label{avoidance-lemma}
Let $t\in [0,\infty) \mapsto M(t)\subset \RR^N$ be
an $m$-dimensional mean curvature flow with boundary.
Suppose 
\begin{gather*}
0 < r < R, \\
\overline{B(0,R)}\cap M(0)
=\varnothing, \\
B(0,R) \cap \partial M(t)=\varnothing \quad\text{for all $t$}. 
\end{gather*}
Then 
\[
\overline{B(0,r)}\cap M(t)
=
\varnothing
\]
for $0\le t \le r(R-r)/m$.
\end{lemma}

\begin{proof}
Consider
\[
  t\in [0, r(R-r)/m]\mapsto K(t):=\overline{B(0, R- \frac{m}rt)}.
\]

Now $M(t)$ is disjoint from $K(t)$ at time $0$, and therefore for all 
 $t\in [0, r(R-r)/m]$, 
as otherwise the maximum principle would be violated
at the first time of contact.
\end{proof}

In the remainder of this section, $C$ will denote an 
$m$-dimensional cone in $\RR^N$ that is smooth except at the origin, its vertex.
Note there is a $\lambda=\lambda_C>0$ with the following property. 
If $p\in C \cap\partial B(0,1)$
and if $\nu$ is a unit normal to $C$ at $p$, then
\[
   C\cap B(p+\lambda\nu,\lambda)
   =\varnothing.
\]
It follows by scaling that 
that if $p\in C\setminus\{0\}$
and if $\nu$ is a unit normal
to $C$ at $p$,
then 
\[
  C\cap B(p+\lambda|p|\,\nu,
              \lambda|p|)
    =
    \varnothing.
\]

\begin{theorem}
\label{eta-theorem}
There exist $\eta>0$
and $R_C<\infty$ (depending 
only on $C$) with the following property.
  Suppose $E$ is a compact expander with
\[
  \partial E\subset C.
\] 
If $p\in E$ and
\[
   |p|\, \dist(p,C) \ge \eta, \tag{*}
\]
then
\[
  |p|\le R_C.
\]
\end{theorem}

\begin{proof}
If, for a given $E$, there is no $p\in E$ such that $(*)$
holds,
then the assertion is vacuously true.  Thus
assume 
\[
  \{p\in E: |p|\,\dist(p,C)\ge \eta\}
\]
is nonempty.   
Then there is a $p$ in that set
for which $|p|$ is greatest.
Of course, it suffices to prove
 that  $|p|\leq R_C$ for that point $p$.

Note that 
\[
   |p|\,\dist(p,C) = \eta,
\]
and thus
\[
   \dist(p,C) = \frac{\eta}{|p|^2} |p|.
\]
Now assume that 
\[
  |p|\ge \sqrt{\eta/2}.
\]
Then 
\[
  \dist(p,C)\le \frac12|p|,
\]
so, if $q$ is a point in $C$
closest to $p$, then
\[
   |q|\ge \frac12|p|.
\]

Now let 
\[
  \nu = \frac{p-q}{|p-q|}
\]
and let
\[
 O' = q + \lambda|q|\nu.
\]

Assume, for the moment, that
\[
   R:=\lambda|q|> r:=\dist(p,C).
\]
Now apply Lemma~\ref{avoidance-lemma}
to $B(0',R)$ and $B(0,r)$.
We see that 
\begin{equation}
\label{expanding}
  t^{1/2}E  
\end{equation}
is disjoint from $\overline{B(0,r})$
for
\[
  t\le r(R-r)/m.
\]
On the other hand,
the surface~\eqref{expanding} at time $1$ (that is, $E$) contains
$p\in \partial B(0',r)$.
Thus 
\[
  1\ge r(R-r)/m,
\]
so
\begin{align*}
m 
&\ge
r(R-r)
\\
&=
\dist(p,C)( \lambda|q| - \dist(p,C))
\\
&\ge
\dist(p,C) 
\left(\frac12\lambda|p|
-
\dist(p,C)
\right)
\\
&=
|p|\,\dist(p,C)
\left(
\frac12\lambda
-
|p|^{-2}|p|\,\dist(p,C)
\right)
\\
&=
\eta
\left(
\frac12\lambda - |p|^{-2}\eta
\right).
\end{align*}
Now assume that 
\[
   |p|^{-2}\eta \le \frac14\lambda,
\]
i.e., that
\[
  |p| \ge 2(\eta/\lambda)^{1/2}.
\]
Then
\[
  m \ge \frac{\eta\lambda}4.
\]
Thus if
\[
   \eta:=\frac{8m}{\lambda},
\]
we get $m\ge 2m$, a contradiction.

To get the contradiction, we assumed that
\[
 |p|\ge \sqrt{\eta/2} = 2(m/\lambda)^{1/2}
\]
and that
\begin{align*}
|p| 
&\ge 
2(\eta/\lambda)^{1/2}
\\
&=
2(8m/\lambda^2)^{1/2}
\\
&=
\frac{4\sqrt{2m}}\lambda.
\end{align*}

Thus the theorem holds for
\[
  \eta=\frac{8m}{\lambda}
\]
and for
\[
  R_C = \max\{ 2(m/\lambda)^{1/2},
\,
   4(2m)^{1/2}/\lambda \}.
\]
\end{proof}

\begin{corollary}
\label{hat-R-corollary}
There is an $\hat R\ge R_C$ such that
if $E$ is a compact expander with $\partial E\subset C$ and if 
 $p\in E\setminus B(0,\hat R)$,
then $p$ is in the interior of the domain of the nearest point
retraction $\pi(\cdot)$
 onto $C$.
\end{corollary}

\begin{proof}
Note that there is a $\delta_C>0$ such that if $p\ne 0$ and
\[
  \dist(p,C) \le \delta_C|p|,
\]
then $p$ is in the interior of the domain of $\pi(\cdot)$.
If $\hat R\ge R_C$ and if $p\in E\setminus B(0,\hat R)$, then
by Theorem~\ref{eta-theorem}, 
\begin{align*}
\dist(p,C)
&\le 
\frac{\eta}{|p|}
\\
&=
\frac{\eta}{|p|^2}{|p|}
\\
&\le \frac{\eta}{\hat R^2}{|p|}.
\end{align*}
Thus it suffices to choose $\hat R\ge R_C$ large enough that $\eta/\hat R^2<\delta_C$.
\end{proof}

\begin{theorem} 
\label{cones-theorem}
Suppose that $C$ is a cone that is smooth except at the origin, and suppose that $0<\eps<\pi/2$.
There exists a $\tilde R <\infty$, depending
only on $C$ and $\eps$, with the 
following property.
If $R>\tilde R$ and if $E$ is an expander in $\overline{B(0,R)}$
such that
\[
 \partial E = C\cap \partial B(0,R),
\]
then
\begin{enumerate}
\item $E\setminus B(0,\tilde R)$
is in the interior of the domain of the nearest point
retraction $\pi(\cdot)$ onto $C$. \label{cone-1}
\item $\pi$ maps 
 $E\setminus B(0,\tilde R)$
 diffeomorphically onto
 a subset of $C$. \label{cone-2}
\item For each 
 $p\in E\setminus B(0,\tilde R)$,
 the angle between $\Tan(E,p)$
and $\Tan(C,\pi(p))$ is 
$< \eps$, and
the angle between $\Tan(E,p)$
and $p$ is $<\eps$. \label{cone-3}
\end{enumerate}
\end{theorem}

\begin{proof}
It suffices to prove the theorem for $\epsilon>0$ sufficiently small, since
the assertion for a smaller value of $\epsilon$ implies the assertion for any
larger value.

Let $\mathcal E$ be the set of all compact expanders $E$ such that
\[
        \partial E = C\cap \partial B(0,R_E)
\]
for some $R_E>0$. For $E\in\mathcal E$, let $\rho(E)$ be the infimum of the
numbers $\rho$ such that the three conclusions of the theorem hold with
$\rho$ in place of $\widetilde R$. The set of such $\rho$ is nonempty, since
the conclusions hold vacuously for $\rho > R_E$, and it is open.
Thus
$\rho(E)$ is not in that set.

The theorem is equivalent to
\[
        \sup_{E\in\mathcal E} \rho(E)<\infty .
\]
Suppose this is false. Then there are $E_n\in\mathcal E$ such that
\[
        \rho_n:=\rho(E_n)\to\infty .
\]
We may assume that $\rho_n>\hat R$, 
where $\hat R$ is as in
Corollary~\ref{hat-R-corollary}.
Thus $E_n\setminus B(0,\rho_n)$ is contained in the interior of 
the domain of $\pi(\cdot)$.

Thus, by definition of $\rho_n$,
there exist 
\[
 p_n\in E_n\cap\partial B(0,\rho_n)
\]
such that one of the following alternatives holds:
\begin{enumerate}[\upshape (i)]
\item 
\label{q-alternative}
there is another point $q_n\in E_n\setminus B(0,\rho_n)$ such that
      \[
            q_n\neq p_n, \qquad \pi(q_n)=\pi(p_n);
      \]

\item
\label{plane-angle-alternative}
the angle between $\operatorname{Tan}(E_n,p_n)$ and
      $\operatorname{Tan}(C,\pi(p_n))$ is at least $\epsilon$;
\item
\label{line-angle-alternative}
the angle between $\operatorname{Tan}(E_n,p_n)$ and the radial direction
      $p_n$ is at least $\epsilon$.
\end{enumerate}

 We wish to handle alternatives~\eqref{q-alternative}, \eqref{plane-angle-alternative}, 
and~\eqref{line-angle-alternative} simultaneously, as much as possible.
For doing that, it is slightly awkward that 
there is a $q_n$ in alternative~\eqref{q-alternative}, but not in alternatives~\eqref{plane-angle-alternative}, 
or~\eqref{line-angle-alternative}.  
We get around that awkwardness by letting $q_n=p_n$ if
alternative~\eqref{q-alternative} does not hold.

Note  that $|q_n|\ge |p_n|$, so
\begin{equation}
\label{qn-is-close}
\begin{aligned}
|q_n-p_n|
&\le |q_n-\pi(p_n)| + |\pi(p_n)-p_n|
\\
&=\dist(q_n,C) + \dist(p_n,C)
\\
&\le \frac{\eta}{|q_n|}
     +
     \frac{\eta}{|p_n|}
\\
&\le \frac{2\eta}{|p_n|}.
\end{aligned}
\end{equation}

Set
\[
        \lambda_n:=|p_n|,\qquad v_n:=\frac{p_n}{|p_n|}.
\]
After passing to a subsequence, $v_n\to v\in \mathbf{S}^{N-1}$. Translate $E_n$, $C$, and  $q_n$ by $-p_n$, and then 
dilate by $\lambda_n$ to get
\begin{align*}
        &E'_n:=\lambda_n(E_n-p_n),
        \\
        &C'_n:=\lambda_n(C-p_n),
        \\
        &q_n':=\lambda_n(q_n-p_n).
\end{align*}
Note that
\[
\dist(0,C_n')
=
\lambda_n\,\operatorname{dist}(p_n,C)\le 
\eta
\]
by Theorem~\ref{eta-theorem},
and  
\[
  |q_n'| =\lambda_n|q_n-p_n| \le 2\eta
\]
by~\eqref{qn-is-close}.
Hence, after passing to a subsequence, the
cones $C'_n$ converge smoothly on compact subsets to an affine $m$-plane
$P$, and the $q_n'$ converge to a point $q'$.
Since $C$ is a cone, the radial direction is tangent to $C$; consequently
$v$ is parallel to $P$.

After passing to a further subsequence, the $E_n'$ converge
to a limit set $E'$.
Note that the rescaled boundaries
$\partial E'_n$ converge to a limit $L$, which is either the empty set or an $(m-1)$-plane
\[
        L=P\cap \{x\cdot v=a\}
\]
for some $a\in [0,\infty)$. In the latter case the convergence of
$\partial E'_n$ to $L$ is smooth with multiplicity one, 
and $E'$ lies in the halfspace
  $\{x\cdot v\leq a\}$.

We now identify the limiting equation. The expander $E_n$ is minimal for the
metric
\[
        c\,\exp\left(\frac{|x|^2}{2m}\right)\delta,
\]
where $c>0$ is arbitrary 
 and $\delta$ is the Euclidean metric.
Under the change of variables
\[
        x=p_n+\lambda_n^{-1}y,
\]
and with the choice
\[
        c_n=\lambda_n^2\exp\left(-\frac{\lambda_n^2}{2m}\right),
\]
the pulled-back metric becomes
\[
        g_n
        =
        \exp\left(
          \frac{|p_n+\lambda_n^{-1}y|^2-\lambda_n^2}{2m}
        \right)\delta
        =
        \exp\left(
          \frac{v_n\cdot y}{m}
          +\frac{|y|^2}{2m\lambda_n^2}
        \right)\delta .
\]
Thus $E_n'$ is $g_n$-minimal.
Note that $g_n$ converges smoothly on compact subsets to the translator metric
\[
        g_\infty=\exp\left(\frac{v\cdot y}{m}\right)\delta .
\]

By Theorem~\ref{eta-theorem}, every point $x\in E_n$ with $|x|\ge R_C$  satisfies
\[
        |x|\,\operatorname{dist}(x,C)\leq \eta .
\]
Therefore, if $K$ is a compact subset
of $\RR^N$, then
\[
        \dist(E'_n\cap K,C'_n)\leq \eta+o(1).
\]
Hence the limit set $E'$  is contained in the slab
\[
        \{x:\dist(x,P)\leq \eta\}.
\]

Let $Q$ be the $m$-plane through the origin parallel to $P$.

\setcounter{claim}{0}
\begin{claim}
\label{smooth-claim}
The $E_n'$ converge smoothly and with multiplicity one to the halfplane 
\[
  Q_a:=P\cap\{x: x\cdot v\le a\}
\]
if $L$ is nonempty, and to the $m$-plane $Q$ if $L$ is empty.
\end{claim}

The proof of Claim~\ref{smooth-claim} is somewhat involved, so first we explain why it implies Theorem~\ref{cones-theorem}.

\begin{proof}[Proof of Theorem~\ref{cones-theorem}, assuming Claim~\ref{smooth-claim}]

By Claim~\ref{smooth-claim},
\[
\angle(\Tan(E_n,p_n), C_n)
\to
\angle(Q, P) = 0
\]
and
\[
\angle(p_n,\Tan(E_n,p_n))
=
\angle(v_n, \Tan(E_n',0))
\to
\angle(v,Q) = 0,
\]
so alternatives~\eqref{plane-angle-alternative} and~\eqref{line-angle-alternative}  do not hold.

Thus alternative~\eqref{q-alternative} holds.
Let $\pi_n$ be the nearest point
retraction to $C_n'$ and 
let $\pi_P:\RR^N\to P$ be the nearest point retraction to $P$.

Since $\pi_n(q_n')=\pi_n(0)$,
we see that $\pi_P(q)=\pi_P(0)$.
Since $q$ and $0$
are both in $Q$, we see that $q=0$.
Thus, by the smooth converges of $E_n'$ to $Q$,
\[
  u_n:= \frac{q_n-0}{|q_n-0|}
\]
converges (after passing to a subsequence) to a unit vector
\begin{equation}
\label{u-in-Q}
u\in Q.
\end{equation}
However, since $\pi_n'$ maps $q_n$ and $0$ to the same point in $C_n'$,
the vector $u_n$ is normal to $C_n'$ at that point.  Thus $u$ is a unit normal to $Q$, and therefore also to $P$, contrary to~\eqref{u-in-Q}.
\end{proof}

Thus it remains only to prove Claim~\ref{smooth-claim}.  
We divide the proof into two cases, according to whether the limit $L$ of the $\partial E_n$ is nonempty or empty.

{\bf Case 1}: $L$ is nonempty.
Note that in this case, $E'$ is contained in the halfspace
\[
    \{x: x\cdot v\le a\}.
\]
By \cite{white-controlling}, $E'$ satisfies the same maximum principles that are satisfied by smooth, $g_\infty$-minimal hypersurfaces properly embedded in
 $\RR^N\setminus L$.
In the language of \cite{white-controlling}, $E'$ is an $(m,0)$-subset of $\RR^N\setminus L$
with respect to~$g_\infty$.

Let $B$ be a bowl soliton disjoint from $L$.  Then (by the strong maximum principle) the set $J$ of $t\ge 0$ such that $B+tv$ is disjoint from $E'$ is an open and closed subset of $[0,\infty)$.  It is nonempty because, for large $t$, $B+tv$ lies in $\{x: x\cdot v>a\}$
and $E'$ lies in $\{x: x\cdot v\le a\}$.  Thus $J=\varnothing$.
It follows that $E'$ is a subset of
\[
  H:=\{x\in P: x\cdot v\le a\},
\]
which is a halfplane bounded by $L$.
By an extension of Allard's Boundary Regularity Theorem
 (see \cite{white-controlling}*{Theorem~6.1}), the $E_n'$ converge smoothly and with multiplicity one to the halfplane $H$.

{\bf Case 2}: $L=\varnothing$.
Recall that 
\[
  E_n\setminus \overline{B(0,\rho_n)}
\]
is the normal graph of a function, with small slopes, over a set in $C$.

It follows that the $E_n'$ converge smoothly (with multiplicity one) on compact subsets of 
\[
   \{x: x\cdot v >0\}
\]
to the graph of a function $f$ on the halfplane
\[
   P\cap\{x: x\cdot v>0\}
\]
such that 
\begin{align*}
\sup |f| &\le \eta, \\
\sup |Df| &\le \tan\eps.
\end{align*}

Let $Z$ be the area blow-up set of the sequence $E'_n$, i.e. the set of points
$q$ such that
\[
        \limsup_{n\to\infty}
        r^{-m}\mathcal H^m(E'_n\cap B(q,r))=\infty
\]
for every $r>0$. By the graphical convergence in $\{x\cdot v>0\}$ and by the
slab bound, we have
\[
        Z\subset \{x\cdot v\leq 0,\ \operatorname{dist}(x,P)\leq \eta\}.
\]
By the standard area blow-up theorem for minimal submanifolds in smoothly
converging metrics \cite{white-controlling}, $Z$ satisfies the maximum principle for the translator
metric $g_\infty$; equivalently, $Z$ is an $(m,0)$-set for $g_\infty$.

We claim that $Z=\varnothing$. 
Let $B$ be a bowl soliton in the halfspace $\{x: x\cdot v > 0\}$.  Then set $J$ of $t\in \RR$ such that $B+tv$ is disjoint from $Z$ is nonempty (it contains $0$) and is open and closed. Thus $J=\RR$, so $Z$ is empty.

Thus the $E'_n$ have uniform local area bounds. After passing to a subsequence,
they converge as varifolds, and smoothly away from the limiting
boundary, to an integral varifold that is stationary for the metric $g_\infty$.  Thus
\begin{equation}
\label{translator}
 t\in \RR \mapsto E'+tv
\end{equation}
is an integral Brakke flow.

Recall that $Q$ is the $m$-plane
through the origin parallel to $P$. 
A tangent flow at $-\infty$ to 
the flow~\eqref{translator}
is supported in $Q$, and thus is $Q$ with some positive integer multiplicity. The graphical description in $\{x\cdot v>0\}$ shows that this
multiplicity is strictly less than $2$, provided $\epsilon$ was chosen
sufficiently small. Hence the multiplicity is one. By Huisken's monotonicity
formula, $E'$ itself is the plane $Q$ with multiplicity one. Allard's
regularity theorem then gives smooth, multiplicity-one convergence of $E_n'$ to $Q$ on compact subsets of $\RR^N$.
This completes the proof 
of Claim~\ref{smooth-claim} and, therefore, the proof of Theorem~\ref{cones-theorem}.
\end{proof}

\begin{remark}
We can generalize Theorem \ref{cones-theorem} to allow for expanders with singularities.
Specifically, let us redefine the class $\Ee$ to be all $E$
 with the following properties:
\begin{enumerate}[(a)]
\item $E$ is a an $m$-dimensional 
integral varifold of bounded
first variation, \label{cone-11}
\item The mean curvature vector $H$
of $E$ satisfies\label{cone-22}
\[
    H(p)= \frac{p^\perp}{2}.
\]
\item The boundary measure of $E$ is $\le \Hh^{m-1}\llcorner(C\cap\partial B(0,R))$
for some $R>0$, and $E$ is supported in the closure of $B(0,R)$. \label{cone-33}
\end{enumerate}

Then Theorem~\ref{cones-theorem} continues to hold, with essentially the same proof.
(In the proof, one adds to alternatives~\eqref{q-alternative}--\eqref{line-angle-alternative} an alternative [0]:  $p_n$ is a singular point of $E_n$.
Also, one uses~\cite{white-controlling}*{Theorem~6.2} in place of~\cite{white-controlling}*{Theorem~6.1}.)
\end{remark}

\begin{theorem}
\label{finite-genus-theorem}
Let $C$ be a cone in $\RR^3$ that is smooth except at the origin. Suppose $E_n\subset \overline{B(0,R_n)}$, with $R_n\rightarrow \infty$, is a sequence of embedded, compact expanders of genus $g$  
 such that
\[
\partial E_n= C\cap B(0,R_n).
\]
Then a subsequence of the $E_n$ converges smoothly to a properly embedded expander of genus $g$ that is asymptotic to $C$.
\end{theorem}

\begin{proof}
By Theorem~\ref{monotonicity-theorem}
\[
  \Hh^2(E_n\cap B(0,r)) \le \Hh^2(C\cap B(0,r))
\]
for every $n$ and $r$.
Thus the areas are uniformly bounded on compact sets.

Hence, by the compactness theorem for embedded minimal surfaces with locally bounded area and genus~\cite{white18}, the $E_n$ converge, 
after passing to a subsequence,
to a smoothly embedded expander $E$, possibly with multiplicity, and the convergence is smooth except at a discrete set of points that lie in the components (if there are any) on which the multiplicity is $>1$.

By Theorem~\ref{cones-theorem}, the multiplicity of $E$ outside of the closure of $B(0,\tilde R)$
is one.  Hence if there were any component of higher multiplicity, it would be a compact expander (without boundary).  But there no such expanders. Hence $E$  has multiplicity one everywhere, and so the convergence of $E_n$ to $E$ is smooth everywhere.

By Theorem~\ref{cones-theorem}, if $R>\tilde R$, then \[
  \genus(E_n\cap B(0,R)) = g
\]
for all $n$. Hence $E$ has genus $g$.
\end{proof}

\begin{theorem}
\label{R-to-infinity-theorem}
Suppose,
 in Theorem~\ref{finite-genus-theorem}, that the cone
satisfies the hypotheses of 
the main theorem, Theorem~\ref{main-theorem},
and suppose that each $E_n$ is $G_k$-symmetric.
Then $E$ is also $G_k$-symmetric.
Furthermore, 
\begin{enumerate}
\item
If each $E_n$ has type 1 (or type 2), then $E$ has type 1 (or type 2).
\item 
If each $E_n$ is big (or small),
then $E$ is big (or small).
\end{enumerate}
\end{theorem}

\begin{proof}
The assertion about big and small follows from Corollary~\ref{epsilon-corollary}.  The other assertions are trivially true.
\end{proof}

\section{The Ilmanen Conjecture} \label{sec:ilmanen}

\begin{theorem}
\label{counterexample-theorem}
For all sufficiently large integers $k$, there
is a $G_k$-invariant, genus $k-1$ shrinker $M$ that is smoothly asymptotic at infinity to a cone $C$
that satisfies the hypotheses
 of Theorem~\ref{main-theorem}.

For each odd prime $p$ and each $n\ge 0$, there exist at least $4$ non-congruent,
$G_k$-invariant expanders $E$
of genus $p^nk-1$ that are asymptotic to $C$ at infinity.  In particular,
\[
t\mapsto
\Sigma(t)
=
\begin{cases}
|t|^{1/2}M  &(t<0),
\\
C &(t=0), 
\\
t^{1/2}E 
&(t>0)
\end{cases}
\]
is a mean curvature flow 
that is smooth except at the spacetime origin, such that the moving surface has genus $k-1$
at times $<0$ and genus $p^nk-1$ at times $>0$.
\end{theorem}

\begin{proof}
The first assertion holds by 
Theorem~\ref{shrinker-theorem}. 
Thus, by Theorem~\ref{main-theorem}, there will be at least four
genus $p^nk-1$ expanders
smoothly asymptotic to $C$
at infinity.
In particular, for each of the two types (Definition~\ref{type-definition}),
there is a both a big expander and a small expander.
By Proposition~\ref{noncongruence-proposition}, 
if two of the four were congruent, they would have to have the same type, and thus would be of different sizes.
But, trivially, no big expander is congruent to a small one.
\end{proof}

Unless $k$ is a power of $2$, we can also prove existence of expanders with less symmetry and with smaller genus than those in 
Theorem~\ref{counterexample-theorem}:

\begin{theorem}
\label{j-theorem}
Suppose that $k$ is large enough that the cone $C_k$ in 
Theorem~\ref{shrinker-theorem}
lies in the region
\[
  |z|\le \eps(|x|^2+|y^2|)^{1/2},
\]
where $\eps$ is as in Theorem~\ref{main-theorem}.
Suppose $k$ is an odd multiple of $j$.
Then there exist at least four $G_j$-invariant expanders of genus $p^nj-1$ that are smoothly asymptotic to $C_k$ at
 infinity.
No two of the four are congruent by an element of $G_j$.
\end{theorem}

The proof is almost identical to the proof of 
Theorem~\ref{counterexample-theorem}.

 \appendix
\section{The Ilmanen  Conjecture for innermost and outermost flows} \label{sec:appendix}

The following theorem, which is a slight extension of Theorems in~\cite{bamler-kleiner}, 
shows that Ilmanen's Genus Reduction Conjecture does hold
for innermost and outermost flows.

(In his conjecture, Ilmanen speaks of a singularity being a ``neckpinch'' or a ``shrinking sphere''.
We interpret ``being a neckpinch'' as ``having a shrinking cylinder tangent flow'' and ``being a shrinking sphere'' as ``having a shrinking sphere tangent flow''.
It seems clear (to us) from the sentences Ilmanen wrote after the conjecture that this was what he meant by those terms.)

\begin{theorem}
\label{inner-outer-theorem}
Suppose that $N$ is a a smooth, complete Riemannian $3$-manifold with Ricci curvature bounded below.
Suppose that $M$ is a compact, two-sided, smoothly embedded surface in $N$.
Let $t\in [0,\infty)\mapsto M(t)$ be the outermost (or the innermost) mean curvature flow with $M(0)=M$.
\begin{enumerate}
\item
\label{inner-outer-one}
Almost all times are regular, i.e., times at which $M(t)$ is a smoothly embedded, multiplicity one surface.
\item 
\label{inner-outer-two}
Suppose that $T>0$ is a regular time for the flow, and suppose that there are at least $k$ spacetime singular points in $N\times [0,T)$ at which there is a tangent flow that is not a shrinking sphere or shrinking cylinder.
Then
\[
 \genus(M(T)) \le \genus(M)
                   - k.
\]
\end{enumerate}
\end{theorem}

Note that $M$ must be two-sided in order for the notions of  innermost flow and outermost flow 
to make sense.

\begin{proof}
We may assume that $M$ is connected, since, if it is not connected, the different components evolve independently of each other.
By lifting to a double cover of $N$, we may assume that $M$ divides $N$ into two connected components.

Assertion~\eqref{inner-outer-one} follows immediately from~\cite{bamler-kleiner}*{Theorem~1.1}.

The proof of 
 Assertion~\eqref{inner-outer-two}
 is a slight modification of the proof 
 of~\cite{bamler-kleiner}*{Theorem~1.9(c)},
 which uses an earlier strict genus-drop theorem  (for ancient solutions) of Chodosh-Choi-Schulze~\cite{ChChS}.

Let $f:N\to \RR$ be the signed distance function to $M$.  Let $\eps>0$ be such that $f$ has no critical values in $[-\eps,\eps]$. Choose $s_k\in (0,\eps)$ converging to $0$ such that 
\[
   M_k:= \{p: f(p)=s_k\}
\]
only develops shrinking sphere and shrinking cylinder singularities under mean curvature flow.

We choose the $s_k$ small enough that 
each $M_k$ is diffeomorphic to $M$.

Let $M_k(t)$ be the result of flowing $M_k$ for time $t$.

Note that if $t$ is a regular time for $M(\cdot)$, then it is a regular time for $M_k(\cdot)$ for all sufficiently large $k$, and 
that $M_k(t)$ converges smoothly to $M(t)$.  In particular, this holds for $t=T$.  Thus
\[
  \genus(M_k(T)) = \genus(M(T))
\]
for all sufficiently large~$k$.
On the other hand, the proof of Claim~7.5 in~\cite{bamler-kleiner} shows that
\[
  \genus(M_k(T)) 
  \le
  \genus(M(0)) - k.
\]
Thus 
\[
\genus(M(T)) < \genus(M)- k.
\]
\end{proof}



\begin{bibdiv}
\begin{biblist}

\bib{ACI}{article}{
  author  = {Angenent, Sigurd B.},
  author={Chopp, David L.},
  author={Ilmanen, Tom},
  title   = {A computed example of nonuniqueness of mean curvature flow in {$\mathbf{R}^3$}},
  journal = {Comm. Partial Differential Equations},
  volume  = {20},
  number  = {11-12},
  year    = {1995},
  pages   = {1937--1958},
  doi     = {10.1080/03605309508821158},
}

\bib{bamler-kleiner}{article}{
      title={On the Multiplicity One Conjecture for Mean Curvature Flows of surfaces}, 
      author={Bamler, Richard},
      author={Kleiner, Bruce},
      year={2023},
      pages={1--58},
      eprint={https://arxiv.org/abs/2312.02106},
}

\bib{barbosa-docarmo}{article}{
  author  = {Barbosa, J. L.},
  author = {do Carmo, M.},
  title   = {On the size of a stable minimal surface in \(R^3\)},
  journal = {American Journal of Mathematics},
  volume  = {98},
  number  = {2},
  year    = {1976},
  pages   = {515--528}
}

\bib{bw21}{article}{
  author  = {Bernstein, Jacob},
  author={Wang, Lu},
  title   = {The space of asymptotically conical expanders of mean curvature flow},
  journal = {Math. Ann.},
  volume  = {380},
  number  = {1-2},
  year    = {2021},
  pages   = {175--230},
  doi     = {10.1007/s00208-021-02147-0},
  review = {\MR{4263682}},
}

\bib{bw21b}{article}{
  author  = {Bernstein, Jacob},
  author={Wang, Lu},
  title   = {Smooth compactness for spaces of asymptotically conical self-expanders of mean curvature flow},
  journal = {Int. Math. Res. Not. IMRN},
  year    = {2021},
  number  = {12},
  pages   = {9016--9044},
  doi     = {10.1093/imrn/rnz087},
  review = {\MR{4276312}},
}

\bib{bw1}{article}{
   author={Bernstein, Jacob},
   author={Wang, Lu},
   title={A mountain-pass theorem for asymptotically conical self-expanders},
   journal={Peking Math. J.},
   volume={5},
   date={2022},
   number={2},
   pages={213--278},
   issn={2096-6075},
   review={\MR{4492654}},
   doi={10.1007/s42543-021-00042-w},
}

\bib{bw2}{article}{
   author={Bernstein, Jacob},
   author={Wang, Lu},
   title={Topological uniqueness for self-expanders of small entropy},
   journal={Camb. J. Math.},
   volume={10},
   date={2022},
   number={4},
   pages={785--833},
   issn={2168-0930},
   review={\MR{4524828}},
   doi={10.4310/cjm.2022.v10.n4.a2},
}

\bib{bw3}{article}{
   author={Bernstein, Jacob},
   author={Wang, Lu},
   title={An integer degree for asymptotically conical self-expanders},
   journal={Calc. Var. Partial Differential Equations},
   volume={62},
   date={2023},
   number={7},
   pages={Paper No. 200, 46},
   issn={0944-2669},
   review={\MR{4621517}},
   doi={10.1007/s00526-023-02541-3},
}

\bib{brendle}{article}{
   author={Brendle, Simon},
   title={Embedded self-similar shrinkers of genus 0},
   journal={Ann. of Math. (2)},
   volume={183},
   date={2016},
   number={2},
   pages={715--728},
   issn={0003-486X},
   review={\MR{3450486}},
   doi={10.4007/annals.2016.183.2.6},
}

\bib{ChChS}{article}{
      title={Mean curvature flow with generic initial data II}, 
      author={Chodosh, Otis},
   author={Choi, Kyeongsu},
      author={Schulze, Felix},
      year={2023},
      journal={arXiv:2302.08409},
      archivePrefix={arXiv},
      primaryClass={math.DG},
      url={https://arxiv.org/abs/2302.08409}, 
}

\bib{ChDHS}{article}{
   author={Chodosh, Otis},
   author={Daniels-Holgate, J. M.},
   author={Schulze, Felix},
   title={Mean curvature flow from conical singularities},
   journal={Invent. Math.},
   volume={238},
   date={2024},
   number={3},
   pages={1041--1066},
   issn={0020-9910},
   review={\MR{4824733}},
   doi={10.1007/s00222-024-01296-8},
}

\bib{ChS}{article}{
   author={Chodosh, Otis},
   author={Schulze, Felix},
   title={Uniqueness of asymptotically conical tangent flows},
   journal={Duke Math. J.},
   volume={170},
   date={2021},
   number={16},
   pages={3601--3657},
   issn={0012-7094},
   review={\MR{4332673}},
   doi={10.1215/00127094-2020-0098},
}

\bib{csch}{article}{
  author  = {Clutterbuck, Julie},
  author={Schn{\"u}rer, Oliver C.},
  title   = {Stability of mean convex cones under mean curvature flow},
  journal = {Math. Z.},
  volume  = {267},
  number  = {3-4},
  year    = {2011},
  pages   = {535--547},
  doi     = {10.1007/s00209-009-0634-4},
}

\bib{Ding2020}{article}{
author = {Ding, Qi},
title = {Minimal cones and self-expanding solutions for mean curvature flows},
journal = {Mathematische Annalen},
volume = {376},
year = {2020},
number = {1-2},
pages = {359--405}
}

\bib{EH}{article}{
  author  = {Ecker, Klaus},
  author={Huisken, Gerhard},
  title   = {Mean curvature evolution of entire graphs},
  journal = {Ann. of Math. (2)},
  volume  = {130},
  number  = {3},
  year    = {1989},
  pages   = {453--471},
  doi     = {10.2307/1971452},
  review = {\MR{1025164}}
}

\bib{grayson}{article}{
   author={Grayson, Matthew A.},
   title={Shortening embedded curves},
   journal={Ann. of Math. (2)},
   volume={129},
   date={1989},
   number={1},
   pages={71--111},
   issn={0003-486X},
   review={\MR{0979601}},
   doi={10.2307/1971486},
}


\bib{gulliver}{article}{
   author={Gulliver, Robert},
   title={Removability of singular points on surfaces of bounded mean
   curvature},
   journal={J. Differential Geometry},
   volume={11},
   date={1976},
   number={3},
   pages={345--350},
   issn={0022-040X},
   review={\MR{0431045}},
}

\bib{hmw-morse}{article}{
   author={Hoffman, David},
   author={Mart\'in, Francisco},
   author={White, Brian},
   title={Morse-Rad\'o{} theory for minimal surfaces},
   journal={J. Lond. Math. Soc. (2)},
   volume={108},
   date={2023},
   number={4},
   pages={1669--1700},
   issn={0024-6107},
   review={\MR{4655275}},
   doi={10.1112/jlms.12791},
}

\bib{hmw-shrinkers}{article}{
      title={Generating Shrinkers by Mean Curvature Flow}, 
      author={Hoffman, David},
      author={Mart\'in, Francisco}, 
      author={White, Brian},
      year={2026},
      eprint={https://arxiv.org/abs/2502.20340},
      archivePrefix={arXiv},
      primaryClass={math.DG},
      url={https://arxiv.org/abs/2502.20340}, 
}

\bib{hoffman-meeks}{article}{
   author={Hoffman, David},
   author={Meeks, William H., III},
   title={Embedded minimal surfaces of finite topology},
   journal={Ann. of Math. (2)},
   volume={131},
   date={1990},
   number={1},
   pages={1--34},
   issn={0003-486X},
   review={\MR{1038356}},
   doi={10.2307/1971506},
}

\bib{hoffman-white-number}{article}{
   author={Hoffman, David},
   author={White, Brian},
   title={On the number of minimal surfaces with a given boundary},
   language={English, with English and French summaries},
   note={G\'eom\'etrie diff\'erentielle, physique math\'ematique,
   math\'ematiques et soci\'et\'e. II},
   journal={Ast\'erisque},
   number={322},
   date={2008},
   pages={207--224},
   issn={0303-1179},
   isbn={978-285629-259-4},
   review={\MR{2521657}}, 
    eprint={https://arxiv.org/abs/0807.0933},
      archivePrefix={arXiv},
      primaryClass={math.DG},
      url={https://arxiv.org/abs/0807.0933}, 
}




\bib{ilmanen1}{article}{
      title={Lectures on Mean Curvature Flow and Related Equations}, 
      author={Ilmanen, Tom},
      year={2026},
      eprint={https://arxiv.org/abs/2601.21952},
      archivePrefix={arXiv},
      primaryClass={math.DG},
      url={https://arxiv.org/abs/2601.21952}, 
}

\bib{huisken}{article}{
   author={Huisken, Gerhard},
   title={Asymptotic behavior for singularities of the mean curvature flow},
   journal={J. Differential Geom.},
   volume={31},
   date={1990},
   number={1},
   pages={285--299},
   issn={0022-040X},
   review={\MR{1030675}},
}

\bib{sz-shrinkers}{article}{
      title={Self-shrinkers with any number of ends in $\mathbb{R}^{3}$ by stacking $\mathbb{R}^{2}$}, 
      author={Shao,Guanhua},
      author={Zou, Jiahua},
      year={2025},
      journal={preprint on arXiv},     
      eprint={https://arxiv.org/abs/2507.18825},
      archivePrefix={arXiv},
      primaryClass={math.DG},
      url={https://arxiv.org/abs/2507.18825}, 
}

\bib{sz-expanders}{article}{
      title={Self-expanders of positive genus}, 
      author={Shao,Guanhua},
      author={Zou, Jiahua},
      year={2025},
      eprint={https://arxiv.org/abs/2507.19428},     
      archivePrefix={arXiv},
      primaryClass={math.DG},
}

\bib{Wang14}{article}{
  author  = {Wang, Lu},
  title   = {Uniqueness of self-similar shrinkers with asymptotically conical ends},
  journal = {J. Amer. Math. Soc.},
  volume  = {27},
  number  = {3},
  year    = {2014},
  pages   = {613--638},
  doi     = {10.1090/S0894-0347-2014-00792-X},
  mrnumber = {3194490},
}


\bib{white-topology}{article}{
   author={White, Brian},
   title={The topology of hypersurfaces moving by mean curvature},
   journal={Comm. Anal. Geom.},
   volume={3},
   date={1995},
   number={1-2},
   pages={317--333},
   issn={1019-8385},
   review={\MR{1362655}},
   doi={10.4310/CAG.1995.v3.n2.a5},
}

\bib{white-controlling}{article}{
   author={White, Brian},
   title={Controlling area blow-up in minimal or bounded mean curvature
   varieties},
   journal={J. Differential Geom.},
   volume={102},
   date={2016},
   number={3},
   pages={501--535},
   issn={0022-040X},
   review={\MR{3466806}},
}



\bib{white18}{article}{
   author={White, Brian},
   title={On the compactness theorem for embedded minimal surfaces in
   3-manifolds with locally bounded area and genus},
   journal={Comm. Anal. Geom.},
   volume={26},
   date={2018},
   number={3},
   pages={659--678},
   issn={1019-8385},
   review={\MR{3844118}},
   doi={10.4310/CAG.2018.v26.n3.a7},
}

\bib{white-degree}{article}{
      author={White, Brian},
       title={Mapping degrees in minimal surface theory},
        date={2026},
      status={in preparation},
}

\end{biblist}

\end{bibdiv}
\end{document}